\documentclass[letterpaper, 11pt, thm-restate]{article}

\usepackage{import}
\newcommand{\figureroadmap}{

\tikzset{every picture/.style={line width=0.75pt}} 

\begin{tikzpicture}[x=0.75pt,y=0.75pt,yscale=-1,xscale=1]

\draw   (54,80) -- (190,80) -- (190,110) -- (54,110) -- cycle ;
\draw   (270,80) -- (410,80) -- (410,110) -- (270,110) -- cycle ;
\draw   (480,80) -- (605.4,80) -- (605.4,110) -- (480,110) -- cycle ;
\draw    (480,95) -- (413,95) ;
\draw [shift={(410,95)}, rotate = 360] [fill={rgb, 255:red, 0; green, 0; blue, 0 }  ][line width=0.08]  [draw opacity=0] (8.04,-3.86) -- (0,0) -- (8.04,3.86) -- (5.34,0) -- cycle    ;
\draw   (400,40) -- (490,40) -- (490,60) -- (400,60) -- cycle ;
\draw    (445,60) -- (445,92) ;
\draw [shift={(445,95)}, rotate = 270] [fill={rgb, 255:red, 0; green, 0; blue, 0 }  ][line width=0.08]  [draw opacity=0] (8.04,-3.86) -- (0,0) -- (8.04,3.86) -- (5.34,0) -- cycle    ;
\draw   (185,40) -- (275,40) -- (275,60) -- (185,60) -- cycle ;
\draw    (230,60) -- (230,92) ;
\draw [shift={(230,95)}, rotate = 270] [fill={rgb, 255:red, 0; green, 0; blue, 0 }  ][line width=0.08]  [draw opacity=0] (8.04,-3.86) -- (0,0) -- (8.04,3.86) -- (5.34,0) -- cycle    ;
\draw    (270,95) -- (193,95) ;
\draw [shift={(190,95)}, rotate = 360] [fill={rgb, 255:red, 0; green, 0; blue, 0 }  ][line width=0.08]  [draw opacity=0] (8.04,-3.86) -- (0,0) -- (8.04,3.86) -- (5.34,0) -- cycle    ;
\draw   (290,140) -- (390,140) -- (390,170) -- (290,170) -- cycle ;
\draw  [draw opacity=0] (140.13,125.14) .. controls (140.09,125.14) and (140.04,125.14) .. (140,125.14) .. controls (131.72,125.14) and (125,118.42) .. (125,110.14) .. controls (125,110.09) and (125,110.05) .. (125,110) -- (140,110.14) -- cycle ; \draw   (140.13,125.14) .. controls (140.09,125.14) and (140.04,125.14) .. (140,125.14) .. controls (131.72,125.14) and (125,118.42) .. (125,110.14) .. controls (125,110.09) and (125,110.05) .. (125,110) ;  
\draw    (140.13,125.14) -- (470.01,125) ;
\draw  [draw opacity=0] (470.01,125) .. controls (478.24,125.06) and (484.9,131.75) .. (484.9,140) .. controls (484.9,148.25) and (478.24,154.94) .. (470,155) -- (469.9,140) -- cycle ; \draw   (470.01,125) .. controls (478.24,125.06) and (484.9,131.75) .. (484.9,140) .. controls (484.9,148.25) and (478.24,154.94) .. (470,155) ;  
\draw    (470,155) -- (393,155) ;
\draw [shift={(390,155)}, rotate = 360] [fill={rgb, 255:red, 0; green, 0; blue, 0 }  ][line width=0.08]  [draw opacity=0] (8.04,-3.86) -- (0,0) -- (8.04,3.86) -- (5.34,0) -- cycle    ;
\draw   (70.27,140) -- (170,140) -- (170,170) -- (70.27,170) -- cycle ;
\draw   (400,182) -- (490,182) -- (490,202) -- (400,202) -- cycle ;
\draw    (445,182) -- (445,157) ;
\draw [shift={(445,154)}, rotate = 90] [fill={rgb, 255:red, 0; green, 0; blue, 0 }  ][line width=0.08]  [draw opacity=0] (8.04,-3.86) -- (0,0) -- (8.04,3.86) -- (5.34,0) -- cycle    ;
\draw    (290,155) -- (173,155) ;
\draw [shift={(170,155)}, rotate = 360] [fill={rgb, 255:red, 0; green, 0; blue, 0 }  ][line width=0.08]  [draw opacity=0] (8.04,-3.86) -- (0,0) -- (8.04,3.86) -- (5.34,0) -- cycle    ;
\draw   (195,175) -- (265,175) -- (265,205) -- (195,205) -- cycle ;
\draw    (230,175) -- (230,158) ;
\draw [shift={(230,155)}, rotate = 90] [fill={rgb, 255:red, 0; green, 0; blue, 0 }  ][line width=0.08]  [draw opacity=0] (8.04,-3.86) -- (0,0) -- (8.04,3.86) -- (5.34,0) -- cycle    ;

\setstretch{0.8}
\draw (55,82) node [anchor=north west][inner sep=0pt]   [align=left] {\begin{minipage}[lt]{100pt}\setlength\topsep{0pt}
\begin{center}
{\footnotesize Theorem~\ref{thm:main}}\\{\footnotesize LE in modules~\cite{dong2025s}}
\end{center}

\end{minipage}};
\draw (266,82) node [anchor=north west][inner sep=0.75pt]   [align=left] {\begin{minipage}[lt]{110pt}\setlength\topsep{0pt}
\begin{center}
{\footnotesize Theorem~\ref{thm:local}, LE in}\\{\footnotesize modules over local rings}
\end{center}

\end{minipage}};
\draw (475,82) node [anchor=north west][inner sep=0.75pt]   [align=left] {\begin{minipage}[lt]{100pt}\setlength\topsep{0pt}
\begin{center}
{\footnotesize Theorem~\ref{thm:DM}}\\{\footnotesize LE in fields~\cite{derksen2012linear}}
\end{center}

\end{minipage}};
\draw (512,114) node [anchor=north west][inner sep=0.75pt]   [align=left] {{\footnotesize (blackbox)}};
\draw (403,44) node [anchor=north west][inner sep=0.75pt]   [align=left] {{\footnotesize Proposition~\ref{prop:internal_rep}}};
\draw (187,44) node [anchor=north west][inner sep=0.75pt]   [align=left] {{\footnotesize Proposition~\ref{prop:decompose_localize}}};

\draw (291,141) node [anchor=north west][inner sep=0.75pt]   [align=left] {\begin{minipage}[lt]{70pt}\setlength\topsep{0pt}
\begin{center}
{\footnotesize LRS in modules}\\{\footnotesize over local rings}
\end{center}

\end{minipage}};
\draw (68,141) node [anchor=north west][inner sep=0.75pt]   [align=left] {\begin{minipage}[lt]{75pt}\setlength\topsep{0pt}
\begin{center}
{\footnotesize Theorem~\ref{thm:LRS}}\\{\footnotesize LRS in modules}
\end{center}

\end{minipage}};
\draw (403,186) node [anchor=north west][inner sep=0.75pt]   [align=left] {{\footnotesize Proposition~\ref{prop:seperable_pol}}};
\draw (195,177) node [anchor=north west][inner sep=0.75pt]   [align=left] {\begin{minipage}[lt]{50pt}\setlength\topsep{0pt}
\begin{center}
{\footnotesize Lemma~\ref{lem:split_extension}}\\{\footnotesize Lemma~\ref{lem:decompose_LRS}}
\end{center}

\end{minipage}};

\end{tikzpicture}

}

\usepackage{setspace}
\usepackage[margin=1in]{geometry}
\usepackage{footmisc}
\usepackage{enumitem}
\usepackage{xcolor}
\usepackage{array}
\usepackage{amsthm}
\usepackage{amsmath}
\usepackage{soul}
\numberwithin{equation}{section}
\usepackage{amssymb}
\usepackage{amsfonts}
\usepackage{graphicx}
\usepackage{thm-restate}
\usepackage[colorlinks = true]{hyperref}
\hypersetup{
    linkcolor=black,
    citecolor=blue,
    urlcolor=blue}
\hypersetup{hypertexnames=false}

\usepackage{mathtools}
\usepackage{mathrsfs}
\usepackage{changepage}

\usepackage{tikz}
\usetikzlibrary{positioning,calc,arrows,automata,shapes,shapes.geometric,decorations.pathmorphing}
\usetikzlibrary{shapes}
\usepackage{tikz-cd}

\allowdisplaybreaks

\usepackage{titlesec}
\titleformat{\section}
  {\normalfont\Large\bfseries}
  {\thesection}
  {0.65em}
  {}
\titleformat{\subsection}[runin]
       {\normalfont\bfseries}
       {\thesubsection}
       {0.5em}
       {}
       [.]

\newcommand{\Z}{\mathbb{Z}}
\newcommand{\N}{\mathbb{N}}
\newcommand{\Q}{\mathbb{Q}}

\newcommand{\F}{\mathbb{F}}
\newcommand{\mZ}{\mathfrak{Z}}
\newcommand{\mI}{\mathcal{I}}

\newcommand{\ba}{\boldsymbol{a}}
\newcommand{\bb}{\boldsymbol{b}}

\newcommand{\bd}{\boldsymbol{d}}
\newcommand{\bs}{\boldsymbol{s}}
\newcommand{\bt}{\boldsymbol{t}}

\newcommand{\bj}{\boldsymbol{j}}
\newcommand{\by}{\boldsymbol{y}}
\newcommand{\bx}{\boldsymbol{x}}
\newcommand{\br}{\boldsymbol{r}}

\newcommand{\bh}{\boldsymbol{h}}

\newcommand{\mP}{\mathcal{P}}

\newcommand{\gen}[1]{\langle {#1} \rangle}

\newcommand{\frp}{\mathfrak{p}}
\newcommand{\frm}{\mathfrak{m}}

\newcommand{\hooklongrightarrow}{\lhook\joinrel\longrightarrow}

\newtheorem{theorem}{Theorem}[section]
\newtheorem{lemma}[theorem]{Lemma}
\newtheorem{proposition}[theorem]{Proposition}
\newtheorem{observation}[theorem]{Observation}

\theoremstyle{definition}
\newtheorem{definition}[theorem]{Definition}
\theoremstyle{definition}
\newtheorem{example}[theorem]{Example}
\theoremstyle{definition}
\newtheorem{remark}[theorem]{Remark}

\begin{document}

\title{Skolem-Mahler-Lech in rings of positive characteristic: a shorter proof and a multi-dimensional generalization}
\author{Ruiwen Dong\footnote{Magdalen College, University of Oxford, United Kingdom, email: ruiwen.dong@magd.ox.ac.uk} \and Doron Shafrir\footnote{Department of Mathematics, Ben-Gurion University of the Negev, Be’er Sheva, Israel}}
\date{ }

\maketitle
\begin{abstract}
    Let $R$ be a commutative ring and $f(a_1, \ldots, a_n) = \sum_{i=1}^k r_{i1}^{a_1} \cdots r_{in}^{a_n} m_i$ be a linear-exponential map over an $R$-module $M$.
    Dong and Shafrir (2026) showed that, when $\ell M = 0$ for some $\ell \in \mathbb{N}_{>0}$, the zero set of $f$ is the intersection of effectively computable $p$-normal sets, where $p$ ranges over the prime divisors of $\ell$.
    This generalizes an earlier theorem of Derksen and Masser (2012) on the solution set of $S$-unit equations over fields of positive characteristic.
    The purpose of this paper is twofold.
    First, we give a shorter proof of Dong and Shafrir's result, using the theorem of Derksen-Masser as a blackbox.
    Our proof also yields a decomposition of the zero set as a positive Boolean combination of affine transformations of zero sets of linear-exponential equations over fields.
    Second, we prove a multi-dimensional generalization of the Skolem-Mahler-Lech theorem over rings of finite characteristic.
    Specifically, we show that the zero set of every $n$-dimensional linear recurrence sequence over an $R$-module $M$ satisfying $\ell M = 0$ is the intersection of effectively computable $p$-normal sets (in $\mathbb{N}^n$), where $p$ ranges over the prime divisors of $\ell$.
    For example, this gives a decision procedure for whether two classical linear recurrence sequences have a common value over a ring of characteristic $p^a$ or $p^a q^b$, where $p$ and $q$ are primes.
\end{abstract}

\section{Introduction}
\subsection{Linear-exponential equation over fields}
Linear-exponential equations have a long history dating back to the early work of Thue, Mahler, and Siegel on Diophantine approximation and transcendental number theory.
They appear throughout mathematics and computer science, including areas such as program analysis, arithmetic theories, analytic combinatorics, and cryptography~\cite{10.5555/1481045, flajolet2009analytic, lipton2022skolem, hieronymi2022strong, ibrahim2024positivity}.
See~\cite{evertse2015unit, evertse2022effective} for monographs on this topic.

For a field $F$ we denote by $F^{\times}$ the set of all non-zero elements of $F$.
A \emph{linear-exponential equation} over $F$ is an equation of the form
\begin{equation}\label{eq:lin_exp_equation}
r_{11}^{a_1} \cdots r_{1n}^{a_n} \cdot m_1 + r_{21}^{a_1} \cdots r_{2n}^{a_n} \cdot m_2 + \cdots + r_{k1}^{a_1} \cdots r_{kn}^{a_n} \cdot m_k = 0,
\end{equation}
where $a_1, \ldots, a_n \in \Z$ are the variables, and $r_{ij} \in F^{\times}, m_i \in F$, are fixed coefficients.
For example, the equation $2^{a} 3^{b} - 4^{a} 5^{b} + \left(\frac{1}{6}\right)^{a} 7^{b} = 0$ is a linear-exponential equation over the field $\Q$.

By setting appropriate coefficients $r_{ij}$ to 1, one can separate variables or include constants in Equation~\eqref{eq:lin_exp_equation}.
For example, $2^a 1^b - 1^a 3^b - 1^a 1^b = 0$ recovers the classical Diophantine equation $2^a - 3^b = 1$.
Therefore, this definition is general enough to cover other variants of linear-exponential form, such as the well-known \emph{S-unit equations} in number theory.


The solution set of Equation~\eqref{eq:lin_exp_equation} is precisely the zero set of the \emph{linear-exponential map} $f$:
\begin{equation}\label{eq:lin_exp_map}
f(\ba) \coloneqq \br_1^{\ba} m_1 + \br_2^{\ba} m_2 + \cdots + \br_k^{\ba} m_k,
\end{equation}
where we denote $\br_i^{\ba} \coloneqq r_{i1}^{a_1} r_{i2}^{a_2} \cdots r_{in}^{a_n}$ for tuples $\br_i = (r_{i1}, \ldots, r_{in})\in (F^{\times})^n$, $\ba = (a_1, \ldots, a_n)\in\Z^n$.

When $F$ is a field of characteristic $0$, the \emph{subspace
theorem} can be used to prove that the zero set 
\[
\mZ(f) \coloneqq \{\ba \in \Z^n \mid f(\ba) = 0\}
\]
of a linear-exponential map $f$ is a finite union of affine subspaces of $\Z^n$~\cite{Evertse1984,PoortenSchlickewei1991}.
However, all known versions of the subspace theorem are ineffective: there is no known algorithm for computing these subspaces, or even for deciding whether $\mZ(f)$ is empty.

When $F$ is a field of characteristic $p > 0$, Derksen and Masser give the following characterization of the zero set $\mZ(f)$:

\begin{restatable}[Linear-exponential equation in fields, Derksen-Masser~\cite{derksen2012linear}]{theorem}{thmfield}\label{thm:DM}
 Let $F$ be a field of characteristic $p > 0$, and let $f(\ba)=\sum_{i = 1}^k\br_i^{\ba} m_i$, where $\br_1, \ldots, \br_k\in (F^{\times})^n, m_1, \ldots, m_k\in F$. 
 Then, the zero set $\mZ(f)$ is effectively $p$-normal.\footnote{This formulation is a minor variant of the statement in~\cite{derksen2012linear}, which considers the S-unit sum $f(\ba_1, \ldots, \ba_k) = \sum_{i = 1}^k\br^{\ba_i} m_i$. 
 Our statement can be obtained by taking $\br$ in this expression to be the concatenation $(\br_1, \ldots, \br_k) \in (F^{\times})^{k n}$, so $f(\ba_{11}, \ba_{12}, \ldots, \ba_{kk}) = \sum_{i = 1}^k\br_1^{\ba_{i1}} \cdots \br_k^{\ba_{ik}} x_i$, then restricting the variables to the diagonal $\ba_{11} = \ba_{22} = \cdots = \ba_{kk} = \ba$, and $\ba_{ij} = 0$ for $i \neq j$.
 This preserves the $p$-normal structure of the zero set (see~Lemma~\ref{lem:preimage_normal}). Similarly, the formulation of Theorem~\ref{thm:main} below is a variant of~\cite{dong2025s}.}
\end{restatable}

Roughly speaking, a \emph{$p$-normal set} is a finite union of subsets of $\Z^n$ whose elements are described by linear expressions involving integer variables and powers of $p$. For example, $\{(x+p^{n}, 2x) \mid x \in \Z, n \in \N\} \subseteq \Z^2$ is a $p$-normal set.
See Definition~\ref{def:pnormal} for the formal definition.
By \emph{effective} we mean that there is an algorithm that computes $\mZ(f)$ from the input $\br_1, \ldots, \br_k, m_1, \ldots, m_k$, under the standard computability assumptions on the field $F$~\cite{baumslag1981computable}.
In particular, this means we can algorithmically decide whether $\mZ(f)$ is empty.

The following classic example illustrates the theorem of Derksen-Masser.

\begin{example}[{\cite[Example~1.3]{derksen2007skolem}}]\label{example:Derksen}
    Let $F$ be the field of rational functions $\F_2(X)$.
    Consider the linear-exponential map $f \colon \Z \rightarrow \F_2(X)$,
    \begin{equation}\label{eq:exampleDerksen}
    f(a) \coloneqq (X+1)^a - X^a - 1^a.
    \end{equation}
    Then the zero set $\mZ(f)$ is $\{2^k \mid k \in \N\}$.
    
    Indeed, if $a = 2^k$ for some $k \in \N$, then applying the identity $(X+1)^2 = X^2 + 1$ (which holds in characteristic $2$), repeatedly $k$ times, we obtain $(X+1)^{2^k} = X^{2^k} + 1$.
    Therefore $f(2^k) = 0$.
    On the other hand, suppose $f(a) = 0$, it is easy to see that $a \geq 0$.
    If $a = 2^k + b$ for some $1 \leq b \leq 2^k-1$, then the expansion of $(X+1)^a = (X^{2^k}+1)(X+1)^b$ must contain more than two monomials, so it cannot be equal to $X^a + 1^a$.
    We conclude that $\mZ(f) = \{2^k \mid k \in \N\}$, which is a 2-normal set.
    \hfill $\blacksquare$
\end{example}

\subsection{Linear-exponential equation over rings and modules}
The definition~\eqref{eq:lin_exp_map} of a linear-exponential map naturally extends from fields to rings and modules.
All rings considered in this paper are commutative.
For a commutative ring $R$, denote by $R^{\times}$ its set of invertible elements.
A \emph{linear-exponential map} over an $R$-module $M$, is a map $f \colon \Z^n \rightarrow M$ of the form
$
    f(\ba)=\sum_{i = 1}^k\br_i^{\ba} m_i
$,
where $\br_1,\ldots,\br_k\in (R^{\times})^n$, and $m_1,\ldots,m_k\in M$.
In particular, if we take $M$ to be the $R$-module $R$, where $R$ is a field, then we recover the definition of linear-exponential maps over fields.

Linear-exponential equations over rings and modules are much more expressive than their field counterparts.
They naturally arise in areas such as symbolic computation, coding theory, and algorithmic group theory, where the underlying structure cannot be embedded into fields~\cite{AugotBardetFaugere2007, dong2025submonoid}. 
For example, the problem of finding sparse polynomials in an ideal $I$ can be formulated as a linear-exponential equation over the $R$-module $R/I$, where $R$ is a polynomial ring~\cite{jensen2017finding}.

When $R$ is a ring of characteristic $0$, the zero set $\mZ(f)$ can be incomputable in general.
More precisely, there are finitely presented modules $M$ over the Laurent polynomial ring $R = \Z[X, X^{-1}]$, where it is undecidable whether a given linear-exponential map $f\colon \Z^n \rightarrow M$ has an empty zero set $\mZ(f)$~\cite{Dong2025LinearEW}.

When $R$ is a ring of characteristic $\ell > 0$, or more generally, when $\ell M = 0$ for $\ell > 0$, Dong and Shafrir~\cite{dong2025s} proved the following generalization of Derksen-Masser (Theorem~\ref{thm:DM}):

\begin{restatable}[Linear-exponential equation in modules~\cite{dong2025s}]{theorem}{thmmain}\label{thm:main}
    Let $R$ be a commutative ring.
    Let $M$ be an $R$-module such that $\ell M=0$ for some $\ell \in\N_{>0}$, and let $f(\ba)=\sum_{i = 1}^k\br_i^{\ba} m_i$, where $\br_1,\ldots,\br_k\in (R^{\times})^n, m_1,\ldots,m_k\in M$.
    Then we can write $\mZ(f) = \bigcap_{p|\ell}S_p$, where $p$ runs over the prime factors of $\ell$ and $S_p$ is effectively $p$-normal.
\end{restatable}

For example, Theorem~\ref{thm:main} shows that the solution set of a linear-exponential equation over the ring $\Z_{/12}[X, X^{-1}, (X+1)^{-1}]$ (viewed as a module over itself) is the intersection of a $2$-normal set and a $3$-normal set.
The word \emph{effective} means that there is an algorithm that computes each component $S_p$ for $p \mid \ell$, assuming standard computability conditions on $R$ and $M$~\cite{baumslag1981computable} (see Section~\ref{subsec:prelim_module}).
However, this does not give a procedure for deciding whether the intersection $\bigcap_{p|\ell}S_p$ is empty: 
this reduces to the existential first-order theory of the structure $\gen{\Z; 0, 1, +, \left(p^{\N}\right)_{p \mid \ell}}$, and is only known to be decidable in special cases such as when $\ell$ has at most two different prime divisors~\cite{karimov2025decidability} or when $n=1$~\cite{dong2025skolemproblemringspositive}.
See Section~\ref{subsec:pnormal_N} for more discussion on this problem.

By the Chinese Remainder Theorem, Theorem~\ref{thm:main} essentially reduces to the case where $\ell = p^e$ for prime $p$ and integer $e \geq 1$.
The main difficulty arises when $e > 1$, as one cannot view $M$ with $pM \neq 0$ as a linear space over a field of characteristic $p$.
Notably, Example~\ref{example:Derksen} fails on the ring $\Z_{/4}[X, X^{-1}, (X+1)^{-1}]$, since $(X+1)^2 \neq X^2+1$ in characteristic $4$.

To prove Theorem~\ref{thm:main}, the approach of~\cite{dong2025s} is to analyse the proof of Theorem~\ref{thm:DM} by Derksen and Masser step by step, and generalize each of its ingredients from fields to modules. 
These generalizations are highly non-trivial and use a combination of automata theory and non-commutative algebra.
Since the proof of the original Derksen-Masser theorem~\cite{derksen2012linear} is already rather intricate, this approach yields a substantially longer and very involved proof.

The first purpose of this paper is to provide a significantly shorter and comparatively elementary proof of Theorem~\ref{thm:main}.
Our approach treats the Derksen-Masser theorem as a \emph{blackbox}, rather than unpacking every component of its proof.
More precisely, we show that the zero set of a linear-exponential map over an $R$-module $M$ can be expressed as a positive Boolean combination of sets of the form $p^t \cdot \mZ(f) + \bb$, where $t \in \N, \bb \in \Z^n$, and $f$ is a linear-exponential map over a \emph{field} of characteristic $p$.
This allows us to apply Derksen-Masser on each $\mZ(f)$.
This explicit decomposition is a new consequence of our approach and does not seem accessible via the proof in~\cite{dong2025s}.
Our key step is Proposition~\ref{prop:internal_rep}, which may be viewed as a linear-exponential analogue of \emph{Hensel lifting}: the classical technique in number theory for lifting solutions of polynomial equations modulo $p$ to modulo $p^e$. 
Beyond its role in the proof, Proposition~\ref{prop:internal_rep} may be of independent interest.
A concrete example (Example~\ref{example:continued}) will be given to illustrate its idea.

\subsection{The Skolem Problem and a multi-dimensional generalization}\label{subsec:intro_Skolem}

A prominent application of linear-exponential equations is the \emph{Skolem Problem}, which asks whether a linear recurrence sequence contains a zero term.
Here, a sequence $\gamma(0), \gamma(1), \gamma(2), \ldots$ is called a \emph{linear recurrence sequence} if there exist $d \geq 1$ and constants $c_1, \ldots, c_{d}$, such that
\begin{equation}\label{eq:recurrence}
    \gamma(n) = c_1 \cdot \gamma(n-1) + \cdots + c_d \cdot \gamma(n-d), \quad \text{ for all } n \geq d.
\end{equation}
Note that $\gamma$ is uniquely determined by the recurrence relation~\eqref{eq:recurrence} as well as its $d$ initial terms $\gamma(0), \gamma(1), \ldots, \gamma(d-1)$.

Over a field of characteristic $0$, decidability of the Skolem Problem  is a longstanding open problem.
The celebrated theorem of Skolem-Mahler-Lech~\cite{skolem1934verfahren, Mahler1935, lech1953note} shows that the zero set of a linear recurrence sequence over a field of characteristic $0$ is a union of a finite set and finitely many arithmetic progressions.
However, the Skolem-Mahler-Lech theorem is not effective, and there is no known algorithm that decides whether the zero set is empty.

Over a field of characteristic $p > 0$, the Skolem Problem is decidable by a result of Derksen:

\begin{theorem}[Skolem-Mahler-Lech in fields of positive characteristic~\cite{derksen2007skolem}]\label{thm:LRS_field}
 Let $F$ be a field of characteristic $p > 0$, and $\gamma \colon \N \rightarrow F$ be a linear recurrence sequence. Then the zero set $\mZ(\gamma)$ is effectively $p$-normal in $\N$.
\end{theorem}
The definition of $p$-normal sets in $\N$ is conceptually similar to $p$-normal sets in $\Z$, but there are some technical nuances: notably, it can \emph{not} be defined as the intersection of $\N$ with a $p$-normal set in $\Z$. The formal definition is given in Definition~\ref{def:pnormal_N}.

Over a ring of characteristic $\ell > 0$, a follow-up paper by Dong and Shafrir~\cite{dong2025skolemproblemringspositive} proved the following generalization of Theorem~\ref{thm:LRS_field}:

\begin{theorem}[Skolem-Mahler-Lech in rings of positive characteristic~\cite{dong2025skolemproblemringspositive}]\label{thm:LRS_ring}
 Let $R$ be a commutative ring of characteristic $\ell > 0$, and $\gamma \colon \N \rightarrow R$ be a linear recurrence sequence. Then we can write $\mZ(\gamma) = \bigcap_{p|\ell}S_p$, where $p$ runs over the prime factors of $\ell$ and $S_p$ is effectively $p$-normal in $\N$.
 Furthermore, it is decidable whether $\bigcap_{p|\ell}S_p$ is empty.
\end{theorem}

The second purpose of our paper is to prove a multi-dimensional generalization of Theorem~\ref{thm:LRS_ring}:

\begin{restatable}[Multi-dimensional Skolem-Mahler-Lech in modules]{theorem}{thmLRS}\label{thm:LRS}
    Let $R$ be a commutative ring.
    Let $M$ be an $R$-module such that $\ell M=0$ for some $\ell \in\N_{>0}$, and let $\gamma \colon \N^n \rightarrow M$ be an $n$-dimensional linear recurrence sequence.
    Then we can write $\mZ(\gamma) = \bigcap_{p|\ell}S_p$, where $p$ runs over the prime factors of $\ell$ and $S_p$ is effectively $p$-normal in $\N^n$.
\end{restatable}

By generalizing to dimensions $n > 1$, we lose the ability to decide whether $\bigcap_{p|\ell}S_p$ is empty.
This reduces to the existential first-order theory of the structure $\gen{\Z; 0, 1, <, +, \left(p^{\N}\right)_{p \mid \ell}}$, which (with the additional $<$) is also known to be decidable when $\ell$ has at most two prime divisors (see Section~\ref{subsec:pnormal_N}).

In Theorem~\ref{thm:LRS}, we define an $n$-dimensional linear recurrence sequence to be a map $\gamma \colon \N^n \rightarrow M$ that satisfies a recurrence in \emph{each} of the $n$ dimensions.
More precisely, there exists an integer $d \geq 1$ and coefficients $c_{ij} \in R, 1 \leq i \leq n, 1 \leq j \leq d$, such that the following recurrence holds for $i = 1, \ldots, n$:
\begin{multline*}
\gamma(a_1, \ldots, a_i, \ldots, a_n) = c_{i1} \cdot \gamma(a_1, \ldots, a_i-1, \ldots, a_n) + \cdots + c_{id} \cdot \gamma(a_1, \ldots, a_i-d, \ldots, a_n), \\
\text{for all $(a_1, \ldots, a_n) \in \N^n$ with $a_i\geq d$}.
\end{multline*}
Such $\gamma$ is uniquely determined by its initial values $\gamma(\ba), \ba \in \{0,1, \ldots, d-1\}^{n}$.

A canonical example of such a sequence is the linear-exponential sum $\gamma(\ba) = \sum_{i = 1}^k \br_i^{\ba} m_i$, where $\ba \in \N^n$ and the entries of $\br_i \in R^n$ are \emph{not necessarily invertible}.
Another interesting example is $\gamma(a_1, a_2) \coloneqq \alpha(a_1) - \beta(a_2)$, where $\alpha, \beta \colon \N \rightarrow M$ are one-dimensional linear recurrence sequences.
In this case, deciding whether $\mZ(\gamma)$ is nonempty is equivalent to determining whether the sets $\{\alpha(k) \mid k \in \N\}$ and $\{\beta(m) \mid m \in \N\}$ have a common element.


We mention that our definition of $n$-dimensional sequences can be equivalently formulated as rational power series of the special form $\frac{h(T_1, \ldots, T_n)}{\prod_{i=1}^n g_i(T_i)}$.
For a rational power series of the general form $\frac{h(T_1, \ldots, T_n)}{g(T_1, \ldots, T_n)}$, and more broadly, for an algebraic power series over a field of characteristic $p$, a result of Adamczewski and Bell~\cite{adamczewski2012vanishing} shows that its vanishing terms form a \emph{$p$-automatic} set.
This class is strictly more general than $p$-normal sets.

To prove Theorem~\ref{thm:LRS}, our main strategy is to reduce it to linear-exponential equations.
Our approach is quite different from the proof of Theorem~\ref{thm:LRS_ring} in~\cite{dong2025skolemproblemringspositive}, which features direct manipulations of power series.
Instead, we first reduce to a setting where the recurrence relations of $\gamma$ satisfy additional assumptions (Proposition~\ref{prop:seperable_pol}), and then derive a linear-exponential expression by inverting the tensor product of the Vandermonde matrices associated to the recurrence polynomials.

\subsection*{Organization of the paper}
In Section~\ref{sec:prelim}, we recall some notions from commutative algebra and define $p$-normal sets in $\Z^n$. 
We also review the preliminary steps of the proof in~\cite{dong2025s} that are common to our approach (Lemma~\ref{lem:inter_normal} and Proposition~\ref{prop:decompose_localize}).
In Section~\ref{sec:short_proof} we give the shorter proof of Theorem~\ref{thm:main} using Derksen-Masser (Theorem~\ref{thm:DM}) as a blackbox.
Our main technical contribution is Proposition~\ref{prop:internal_rep}.
In Section~\ref{sec:multi_dim} we define $p$-normal sets in $\N^n$ and give the proof of Theorem~\ref{thm:LRS}.
Most technical proofs are accompanied by concrete examples.
See Figure~\ref{fig:roadmap} for a dependency graph between the main statements in the paper.

\begin{figure}[h!]
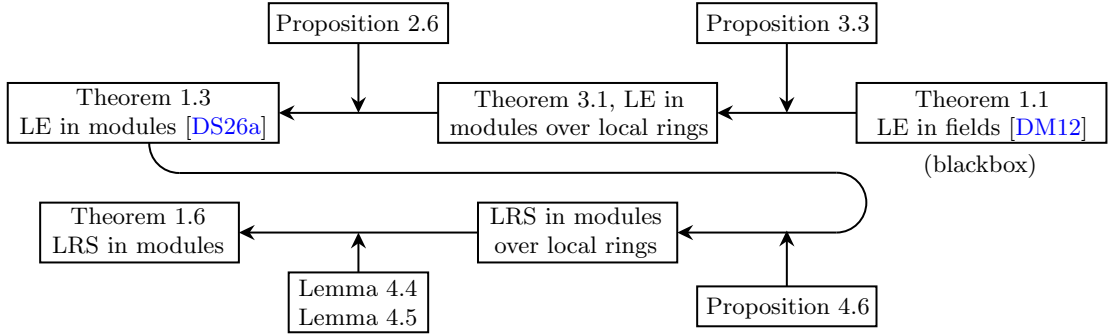

    \centering
    \figureroadmap
    \caption{Dependency graph between the main statements. LE stands for linear-exponential equations, LRS stands for (multi-dimensional) linear recurrence sequences.}
    \label{fig:roadmap}
\end{figure}

\section{Preliminaries}\label{sec:prelim}
In this paper, all rings are commutative and unital (contain $1$).
In particular, any integer $\ell$ can be considered as the element $\ell \cdot 1$ of the ring.
We adopt the convention that $0 \in \N$.

\subsection{Modules and ideals}\label{subsec:prelim_module}

Let $R$ be a ring.
An \emph{$R$-module} is defined as an abelian group $(M, +)$ along with an operation $\cdot \;\colon R \times M \rightarrow M$, satisfying $r \cdot (m+m') = r \cdot m + r \cdot m'$, $(r + s) \cdot m = r \cdot m + s \cdot m$, $rs \cdot m = r\cdot (s \cdot m)$, and $1 \cdot m = m$.
The neutral element in $M$ is denoted as $0_M$ or simply $0$.
For example, for any $d \in \N$, $R^d$ is an $R$-module by $s \cdot (r_1, \ldots, r_d) = (sr_1, \ldots, sr_d)$.
An \emph{ideal} of $R$ is an $R$-submodule of $R$.

Given $m_1, \ldots, m_k$ in an $R$-module $M$, let 
\[
\textstyle
\gen{m_1, \ldots, m_k} \coloneqq \left\{\sum_{i=1}^k r_i \cdot m_i \;\middle|\; r_1, \ldots, r_k \in R\right\}
\]
denote the $R$-submodule generated by $m_1, \ldots, m_k$.
For an $R$-submodule $N \subseteq M$ and $m_1, m_2 \in M$, we write $m_1 \equiv m_2 \pmod N$ if $m_1 - m_2 \in N$.
Define the quotient $M/N \coloneqq \{\overline{m} \mid m \in M\}$ where $\overline{m_1} = \overline{m_2}$ if and only if $m_1 \equiv m_2 \pmod N$.
This quotient is also an $R$-module.
For a finitely generated ring $R$, every finitely generated $R$-module can be written as a quotient $R^d/\gen{v_1, \ldots, v_k}$ for some $d \in \N$ and $v_1, \ldots, v_k \in R^d$. 

Throughout this paper, we assume the standard computability conditions~\cite{baumslag1981computable} on the rings $R$ and the $R$-modules $M$.
Specifically, given elements $r_1, \ldots, r_n \in R$ and $m_1, \ldots, m_d \in M$, we assume one can compute:
\begin{enumerate}[nosep, label=(\roman*)]
    \item a finite generating set for the ideal $\{f \in \Z[X_1, \ldots, X_n] \mid f(r_1, \ldots, r_n) = 0_R\}$, and
    \item a finite generating set for the $\widetilde{R}$-module $\{(s_1, \ldots, s_d) \in \widetilde{R}^d \mid \sum_{i = 1}^d s_i m_i = 0_M\}$, where $\widetilde{R}$ is the subring generated by $r_1, \ldots, r_n$.
\end{enumerate}
These assumptions are satisfied in most settings of literature and applications, for example, when $R$ and $M$ are given by finite presentations $R = \Z[X_1, \ldots, X_n]/\gen{f_1, \ldots, f_k}, M = R^d/\gen{v_1, \ldots, v_k}$.


\medskip
Given an ideal $I \subseteq R$ and an $R$-module $M$, denote by $IM$ the submodule generated by the elements $im, i \in I, m \in M$.
Similarly, for $t \in \N$, denote by $I^t$ the ideal generated by the products $i_1 \cdots i_t$, where $i_1, \ldots, i_t \in I$.

We begin with the following basic lemma, which will be used several times throughout the paper.
\begin{lemma}[{cf.~\cite[Lemma~3.13]{dong2025s}}]\label{lem:lifting}
    Let $R$ be a ring and $p \geq 1$ be an integer.
    Let $I \subseteq R$ be an ideal such that $p \in I$, then for any $r, s \in R$ and $t \in \N$, we have
    \[
    r \equiv s \pmod{I} \implies r^{p^{t-1}} \equiv s^{p^{t-1}} \pmod{I^t}.
    \]
\end{lemma}
\begin{proof}
    We claim that for every $k \geq 1$, $x \equiv y \pmod{I^k} \implies x^p \equiv y^p \pmod{I^{k+1}}$.
    Indeed, suppose $x \equiv y \pmod{I^k}$.
    Then $x = y + z$ for some $z \in I^k$, and hence $x^p = (y + z)^p = y^p + \sum_{i = 1}^p \binom{p}{i} y^{p-i} z^i$. 
    For $i = 2, \ldots, p$, we have $z^i \in I^{ki} \subseteq I^{k+1}$; while for $i = 1$, we have $\binom{p}{i} z^i = pz \in p \cdot I^k \subseteq I^{k+1}$.
    In both cases we have $\binom{p}{i} y^{p-i} z^i \in I^{k+1}$.
    Therefore $x^p - y^p = \sum_{i = 1}^p \binom{p}{i} y^{p-i} z^i \in I^{k+1}$.

    Applying this claim successively for $k = 1, 2, \ldots, t-1$, we obtain
    $
    r \equiv s \pmod{I} \implies r^p \equiv s^p \pmod{I^2} \implies r^{p^2} \equiv s^{p^2} \pmod{I^3} \implies \cdots \implies r^{p^{t-1}} \equiv s^{p^{t-1}} \pmod{I^{t}}
    $.
\end{proof}

For example, taking $I = \gen{2}$ in $\Z$ we obtain $r \equiv s \pmod{2} \implies r^{2^{t-1}} \equiv s^{2^{t-1}} \pmod{2^t}$.

\subsection{p-normal sets in $\Z^n$}

The following definition of $p$-normal sets is a reformulation of~\cite{derksen2012linear}.

\begin{definition}[$p$-normal sets in $\Z^n$]\label{def:pnormal}
    First, a set $S \subseteq \Z^n$ is called \emph{elementary $p$-nested} if it is of the form
    \begin{equation}\label{eq:pnested_Zn}
    \left\{\ba_0 + p^{\ell k_1} \ba_1 + \cdots + p^{\ell k_r} \ba_r \;\middle|\; k_1, k_2, \ldots, k_r \in \N \right\},
    \end{equation}
    where $\ell \geq 1$ and $\ba_0, \ba_1, \ldots, \ba_r \in \Q^n$.
    Note that the entries of $\ba_i$'s do not have to be integers: for example, the set $\left\{\frac{1}{2} + 3^k \cdot \frac{1}{2}\;\middle|\; k \in \N\right\} \subseteq \Z$ is elementary $3$-nested.
    A singleton is by definition elementary $p$-nested with $r = 0$.

    Next, a set $S \subseteq \Z^n$ is called \emph{$p$-succinct}, if it is of the form $H + D \coloneqq \{\bh + \bd \mid \bh \in H, \bd \in D\}$, where $H$ is a subgroup of $\Z^n$ and $D$ is elementary $p$-nested.
    For example, the set $\{(x+p^{k}, 2x) \mid x \in \Z, k \in \N\} \subseteq \Z^2$ is $p$-succinct as it is the sum of the subgroup $H = \{(x, 2x) \mid x \in \Z\} \leq \Z^2$ and the elementary $p$-nested set $D = \{(1, 0) \cdot p^{k} \mid k \in \N\}$.

    Finally, a set $S \subseteq \Z^n$ is called \emph{$p$-normal} if it is a finite union of $p$-succinct sets.
    The empty set is by definition $p$-normal.
    A set is called \emph{effectively} $p$-normal if all its defining coefficients can be effectively computed.
    \hfill $\blacksquare$
\end{definition}

It is easy to see that if $\varphi \colon \Z^n \rightarrow \Z^m$ is a $\Z$-linear map and $S \subseteq \Z^n$ is $p$-normal, then the image $\varphi(S) \subseteq \Z^m$ is also $p$-normal.
The same is true for preimages:



\begin{restatable}[\cite{derksen2015linear}]{lemma}{lempreimage}\label{lem:preimage_normal}
    Let $\varphi \colon \Z^m \rightarrow \Z^n$ be a $\Z$-linear map and $S \subseteq \Z^n$ be a $p$-normal set. Then $\varphi^{-1}(S)$ is effectively $p$-normal.
\end{restatable}
\begin{proof}
    Proposition~2.1 of~\cite{derksen2015linear} states that the preimage $\varphi^{-1}(T)$ of a $p$-succinct set $T \subseteq \Z^n$ is effectively $p$-normal.\footnote{Write $T=B+U$, where $B$ is a subgroup of $\Z^n$ and $U\subseteq\Z^n$ is an elementary $p$-nested set. Proposition~2.1 of~\cite{derksen2015linear} shows that $\varphi^{-1}(T)$ is $p$-normal when $U$ is defined by integer coefficients $\ba_0,\ldots,\ba_r$, whereas our definition allows these coefficients to be rational. This causes no loss of generality: multiplying $\varphi$, $B$, and $U$ by a common denominator of $\ba_0,\ldots,\ba_r$ yields an equivalent instance in which the coefficients are integral.}
    Since every $p$-normal set is a finite union of $p$-succinct sets, it follows that the preimage $\varphi^{-1}(S)$ of a $p$-normal set $S$ is effectively $p$-normal.
\end{proof}

We can use this fact to show the following:

\begin{lemma}\label{lem:inter_normal}
    Given two $p$-normal sets $S, T \subseteq \Z^n$, the intersection $S \cap T$ is effectively $p$-normal.
\end{lemma}
\begin{proof} 
    We give here a short proof using Lemma~\ref{lem:preimage_normal}.
    A longer, self-contained proof was also given in~\cite[Proposition~3.7]{dong2025s}.
    
    Define $S \times T = \{(\bs, \bt) \mid \bs \in S, \bt \in T\} \subseteq \Z^{2n}$.
    By writing $S \times T$ as the Minkowski sum of the two $p$-normal sets $S \times \{0^n\}$ and $\{0^n\} \times T$, we can see that $S \times T$ is $p$-normal.
    Indeed, by definition, the sum of two elementary $p$-nested sets is elementary $p$-nested, and the sum of two subgroups is a subgroup.
    It follows that the sum of two $p$-normal sets is $p$-normal.
    Define the diagonal map $\varphi \colon \Z^n \rightarrow \Z^{2n}, \ba \mapsto (\ba, \ba)$, then the intersection $S \cap T$ is equal to the preimage $\varphi^{-1}(S \times T)$, so it is $p$-normal by Lemma~\ref{lem:preimage_normal}.
\end{proof}

\subsection{Reduction to local rings}
For a standard reference on commutative algebra, see~\cite{atiyah1969introduction}.
For any $n \geq 1$, we denote by $\Z_{/n}$ the ring of integers modulo $n$.

Let $R$ be a ring.
An ideal $\frp \subset R$ is called \emph{prime} if $\frp \neq R$ and for every $a, b \in R$, we have $ab \in \frp \implies a \in \frp \,\text{ or }\, b \in \frp$.
An ideal $\frm \subset R$ is called \emph{maximal} if $\frm \neq R$ and there is no ideal $J$ such that $\frm \subsetneq J \subsetneq R$.
Every maximal ideal is prime.

A ring $R$ is called \emph{local} if it has exactly one maximal ideal.
For example, for any $n$, the ring $\Z_{/2^n}$ is local, since its ideals are of the form $\gen{2^k}, 0 \leq k \leq n$, with $\gen{2}$ the unique maximal ideal.
However the ring $\Z_{/6}$ is not local since both $\gen{2}, \gen{3}$ are maximal.
As a special case, fields are local rings with maximal ideal $\{0\}$.

Let $R$ be a local ring with maximal ideal $\frm$.
Then an element $x \in R$ is invertible if and only if $x \notin \frm$.
In other words, $R^{\times} = R \setminus \frm$.
The quotient ring $R/\frm$ is a field, called the \emph{residue field} of $R$.

\begin{example}\label{example:Z4}
    We give an example of a local ring that will appear in later examples.
    Let $\Z_{/4}[X]$ be the set of univariate polynomials with coefficients in $\Z_{/4}$.
    Consider the ring whose elements are fractions $\frac{f}{g}$, where $f, g \in \Z_{/4}[X]$ and $2 \nmid g$, with equality $\frac{f_1}{g_1} = \frac{f_2}{g_2}$ iff $f_1 g_2 = f_2 g_1$.
    The ring operations are the natural addition and multiplication on fractions: $\frac{f_1}{g_1} + \frac{f_2}{g_2} = \frac{f_1 g_2 + f_2 g_1}{g_1g_2}$, $\frac{f_1}{g_1} \cdot \frac{f_2}{g_2} = \frac{f_1 f_2}{g_1g_2}$.
    We denote this ring by $\Z_{/4}(X)$, which is a non-standard notation:
    \[
    \Z_{/4}(X) \coloneqq \left\{\frac{f}{g} \;\middle|\; f, g \in \Z_{/4}[X],\; 2 \nmid g \right\}.
    \]
    It is important that we do not allow denominators that are divisible by $2$.
    This is because $\frac{1}{2}$ is not well-defined, otherwise $\frac{1}{2} \cdot \frac{1}{2} = \frac{1}{4} = \frac{1}{0}$.
    
    Then, $\Z_{/4}(X)$ is a local ring, its maximal ideal is $\gen{2} = \left\{\frac{2f}{g} \;\middle|\; f, g \in \Z_{/4}[X],\; 2 \nmid g \right\}$, and its residue field is $\Z_{/4}(X) / \gen{2} = \F_2(X)$.
    In particular, $\frac{f}{g} \in \Z_{/4}(X)$ is invertible if and only if $2 \nmid f$.
    There is a natural injective map $\Z_{/4}[X] \hookrightarrow \Z_{/4}(X),\; f \mapsto \frac{f}{1}$.
    Also, $\gen{2}^t=0$ for all $t \geq 2$.
    \hfill $\blacksquare$
\end{example}

Using a decomposition-localization argument, one can express the zero set $\mZ(f)$ of a linear-exponential map $f \colon \Z^n\rightarrow M$ as a finite intersection $\bigcap_{i = 1}^d \mZ(f_i)$, where each $f_i \colon \Z^n\rightarrow M_i$ is a linear-exponential map to a module $M_i$ over a local ring $R_i$:

\begin{proposition}[{\cite[Section~9]{derksen2007skolem},~\cite[Section~3.2]{dong2025s}}]\label{prop:decompose_localize}
    Let $f \colon \Z^n\rightarrow M$ be a linear-exponential map over some $R$-module $M$.
    Then one can effectively compute linear-exponential maps $f_i:\Z^n\rightarrow M_i, i = 1, \ldots, d$, such that $\mZ(f) = \bigcap_{i = 1}^d \mZ(f_i)$, where each $M_i$ is a finitely generated module over a local ring $R_i$ with maximal ideal $\frm_i$, such that $\frm_i^t = 0$ for some $t \in \N$.
    Moreover, if $\ell M = 0$ for some $\ell \in \N_{>0}$, then for each $i$, the residue field $R_i/\frm_i$ has characteristic $p_i$, where $p_i$ is a prime divisor of $\ell$.
\end{proposition}
\noindent Before giving the proof, we recall the following facts from commutative algebra~\cite{eisenbud2013commutative}.
    \begin{enumerate}[nosep, label=(\roman*)]
        \item Let $R$ be a ring. For a prime ideal $\frp\subset R$, the localization of $R$ at $\frp$ is $R_{\frp}=(R\setminus\frp)^{-1}R$.
        Its elements are fractions $\frac{r}{s}$ with $r\in R$ and $s\in R\setminus\frp$, where $\frac{r_1}{s_1} = \frac{r_2}{s_2}$ if and only if $s(s_1 r_2 - s_2 r_1) = 0$ for some $s \in R \setminus \frp$.
        The ring $R_{\frp}$ is local with maximal ideal $\frp R_{\frp}$.
        Similarly, for an $R$-module $M$, we define $M_{\frp} = (R \setminus \frp)^{-1} M$.
        It is an $R_{\frp}$-module with elements $\frac{m}{s}, m \in M, s \in R \setminus \frp$, where $\frac{m_1}{s_1} = \frac{m_2}{s_2}$ if and only if $s(s_1 m_2 - s_2 m_1) = 0$ for some $s \in R \setminus \frp$.
        \item An $R$-module $M$ is called \emph{coprimary} if $rm = 0 \implies (m = 0) \,\text{ or }\, (r^nM=0$ for some $n \geq 1)$.
        In this case, $\frp \coloneqq \{r \in R \mid \exists n \geq 1, r^nM=0\}$ is a prime ideal of $R$, and $M$ is called $\frp$-coprimary.
        When $M$ is $\frp$-coprimary, the localization map $\phi \colon M\rightarrow M_\frp, m \mapsto \frac{m}{1}$ is injective.
        Indeed, if $sm=0$ for some $s\in R \setminus \frp$, then coprimaryness implies either $m=0$ or $s^nM=0$ for some $n\geq1$. The latter would imply $s\in\frp$, a contradiction.
        \item A \emph{primary decomposition} of a submodule $N \subseteq M$ is an expression $N = \bigcap_{i = 1}^d N_i$, where each $N_i$ is a submodule of $M$ such that $M/N_i$ is coprimary.
        The Lasker-Noether theorem~\cite[Theorem~3.10]{eisenbud2013commutative} states that such a decomposition always exists when $R$ is Noetherian and $M$ is finitely generated.
        In this case, a primary decomposition can be effectively computed~\cite{rutman1992grobner}.
    \end{enumerate}
\begin{proof}[Proof of Proposition~\ref{prop:decompose_localize}]
    An elementary illustration of the proof is given in Example~\ref{example:decompose} below.
    The main idea is as follows.
    We use primary decomposition to obtain an injection
    $
    M \hooklongrightarrow M_1 \times \cdots \times M_d,
    $
    where each $M_i$ is a $\frp_i$-coprimary $R$-module for some prime ideal $\frp_i$.
    If we take $f_i$ to be the composition of $f$ with the projection $M \rightarrow M_i$, then $\mZ(f) = \bigcap_{i = 1}^d \mZ(f_i)$.
    It remains to replace $R_i$ and $M_i$ by suitable localizations and quotients so that they satisfy the conditions required in the proposition.
    
    We now give the proof in detail. 
    Write $f(\ba) = \sum_{i=1}^k \br_i^{\ba} m_i$.
    By restricting $R$ to the subring generated by the entries of the $\br_i$'s and their inverses, we can suppose the ring $R$ to be finitely generated over $\Z$ and hence Noetherian.
    Then by restricting $M$ to the submodule generated by the $m_i$'s, we can suppose $M$ to be finitely generated as an $R$-module.
    
    Let $\{0\} = \bigcap_{i = 1}^d N_i$ be the primary decomposition of the zero submodule of $M$.
    Then, for each $i$, $M/N_i$ is $\frp_i$-coprimary for some prime ideal $\frp_i \subset R$.
    Let $\pi_i \colon M\rightarrow M/N_i$ be the canonical projection, then $\bigcap_{i = 1}^d \ker(\pi_i) = \{0\}$, so we have $\mZ(f)=\bigcap_{i = 1}^d \mZ(\pi_i\circ f)$.
    
    We now consider each individual $i$.
    Replacing $M$ by $M/N_i$ and $f$ by $\pi_i\circ f$, we may assume $M$ is $\frp$-coprimary. 
    Since $M$ is $\frp$-coprimary, there exists $t \in \N$ such that $\frp^tM=0$ (see~\cite[Proposition~3.9]{eisenbud2013commutative}). 
    Let $R_{\frp},M_\frp$ denote the localizations of $R, M$ at $\frp$.
    Then $R_{\frp}$ is a local ring with maximal ideal $\frm \coloneqq \frp R_{\frp}$, and $M_{\frp}$ is a finitely generated $R_{\frp}$-module such that $\frm^t M_{\frp} = 0$.
    Since $\frm^t M_{\frp} = 0$, $M_{\frp}$ is naturally a finitely generated module over the quotient ring $\overline{R_{\frp}} \coloneqq R_{\frp}/\frm^t$.
    This is a local ring with maximal ideal $\overline{\frm} \coloneqq \frm \overline{R_{\frp}}$, such that $\overline{\frm}^t = 0$.
    
    Let $\varphi \colon R \rightarrow \overline{R_{\frp}}$ be the composition of the localization map with the quotient map, and let $\phi \colon M \rightarrow M_{\frp}$ be the localization map.
    Since $M$ is $\frp$-coprimary, $\phi$ is injective (see (ii) above).
    We define $f_{\frp}=\phi\circ f$. 
    Then $f_{\frp}(\ba) = \sum_{i=1}^k \varphi(\br_i)^{\ba} \phi(m_i)$, so $f_{\frp}$ is a linear-exponential map over the $\overline{R_{\frp}}$-module $M_{\frp}$. 
    Since $\phi$ is injective, we have $f(\ba) = 0 \iff f_{\frp}(\ba) = 0$.
    This proves the first statement of the proposition by taking $f_i \coloneqq f_{\frp}$, $M_i \coloneqq M_{\frp}$, $R_i \coloneqq \overline{R_{\frp}}$, and $\frm_i \coloneqq \overline{\frm}$.

    For the second statement, since $\ell M = 0$, we have $\ell (M/N_i) = 0$.
    Hence $\ell \in \frp_i$ for $i = 1, \ldots, d$.
    Since $\frp_i$ is prime and $\ell$ is a product of prime numbers, there exists a prime divisor $p_i\mid\ell$ such that $p_i \in \frp_i$.
    Therefore $p_i \in \frm_i$ and the residue field $R_i/\frm_i$ has characteristic $p_i$.
\end{proof}

\begin{example}\label{example:decompose}
    Consider the $\Z[X, Y]$-module\footnote{For expository convenience, we are using a ring $\Z[X, Y]$ without non-trivial invertible elements.
    To obtain a proper example, one can for example replace $\Z[X, Y]$ with the Laurent polynomial ring $\Z[X, X^{-1}, Y, Y^{-1}]$.} $M = \Z[X, Y]/\gen{12, Y^2-3XY+6X^2-4X}$, then $12M = 0$.
    For a map $f \colon \Z^n \rightarrow M$, we want to write $\mZ(f) = \bigcap_{i = 1}^d \mZ(f_i)$ where each $f_i$ maps to a module $M_i$ over a local ring.
    
    Since $12 = 3 \times 4$, we have $\{0\} = 3M \cap 4M$.
    So by the Chinese Remainder Theorem we can decompose $M$ as a direct product $(M/3M) \times (M/4M)$:
    \[
    M = \Big(\underbrace{\Z[X, Y]/\gen{3, Y^2-3XY+6X^2-4X}}_{M/3M}\Big) \times \Big(\underbrace{\Z[X, Y]/\gen{4, Y^2-3XY+6X^2-4X}}_{M/4M}\Big).
    \]
    The first component can be simplified as 
    \[
    \Z[X, Y]/\gen{3, Y^2-3XY+6X^2-4X} = \Z[X, Y]/\gen{3, Y^2-X} = \F_3[X, Y]/\gen{Y^2-X}.
    \]
    The second component can be further decomposed:
    \begin{multline*}
        \Z[X, Y]/\gen{4, Y^2-3XY+6X^2-4X} = \Z[X, Y]/\gen{4, Y^2-3XY + 2X^2} \\
        = \Z_{/4}[X, Y]/\gen{(Y-X)(Y-2X)}
        = \Big(\underbrace{\Z_{/4}[X, Y]/\gen{Y-X}}_{M/(4M+(Y-X)M)}\Big) \times \Big(\underbrace{\Z_{/4}[X, Y]/\gen{Y-2X}}_{M/(4M+(Y-2X)M)}\Big).
    \end{multline*}
    Thus we obtain the decomposition of $M$ into coprimary modules:
    \begin{equation}\label{eq:decompose_344}
    M = \Big(\underbrace{\F_3[X, Y]/\gen{Y^2-X}}_{M/3M}\Big) \times \Big(\underbrace{\Z_{/4}[X, Y]/\gen{Y-X}}_{M/(4M+(Y-X)M)}\Big) \times \Big(\underbrace{\Z_{/4}[X, Y]/\gen{Y-2X}}_{M/(4M+(Y-2X)M)}\Big).
    \end{equation}
    In the first component of~\eqref{eq:decompose_344}, via the variable substitution $X = Y^2$, we have an injection
    \begin{equation}\label{eq:inject_1}
    \F_3[X, Y]/\gen{Y^2-X} \xrightarrow[X \mapsto Y^2]{\sim} \F_3[Y] \hooklongrightarrow \F_3(Y).
    \end{equation}
    In the second component of~\eqref{eq:decompose_344}, via the variable substitution $Y = X$, we have an injection
    \begin{equation}\label{eq:inject_2}
    \Z_{/4}[X, Y]/\gen{Y-X} \xrightarrow[Y \mapsto X]{\sim} \Z_{/4}[X] \hooklongrightarrow \Z_{/4}(X).
    \end{equation}
    (Recall the definition of the local ring $\Z_{/4}(X)$ from Example~\ref{example:Z4}.)
    In the third component of~\eqref{eq:decompose_344}, via the variable substitution $Y = 2X$, we have an injection 
    \begin{equation}\label{eq:inject_3}
    \Z_{/4}[X, Y]/\gen{Y-2X} \xrightarrow[Y \mapsto 2X]{\sim} \Z_{/4}[X] \hooklongrightarrow \Z_{/4}(X).
    \end{equation}
    To summarize, we have the following ``decomposition and localization'' diagram:
    \[
    \begin{tikzcd}[column sep=0mm]
    M \;\; = & \Big(M/3M\Big) \arrow[d, hook] & \times \;\; \Big(M/(4M+(Y-X)M)\Big) \arrow[d, hook] & \times \;\; \Big(M/(4M+(Y-2X)M)\Big), \arrow[d, hook] \\
    & \underbrace{\F_3(Y)}_{M_1} & \underbrace{\Z_{/4}(X)}_{M_2} & \underbrace{\Z_{/4}(X)}_{M_3}
    \end{tikzcd}
    \]
    where $\F_3(Y)$, $\Z_{/4}(X)$, $\Z_{/4}(X)$ are local rings considered as modules over themselves.
    Using the injection $M \hookrightarrow M_1 \times M_2 \times M_3$, the linear-exponential map $f\colon \Z^n \rightarrow M$ is projected on three linear-exponential maps $f_1 \colon \Z^n \rightarrow M_1,\; f_2 \colon \Z^n \rightarrow M_2,\; f_3 \colon \Z^n \rightarrow M_3$, and $f(a)=0$ if and only if $f_1(a) = f_2(a) = f_3(a) = 0$.
    Therefore $\mZ(f) = \mZ(f_1) \cap \mZ(f_2) \cap \mZ(f_3)$.
    \hfill $\blacksquare$
\end{example}

\section{A proof via the Derksen-Masser theorem}\label{sec:short_proof}
In this section we give a shorter, comparably elementary proof of Theorem~\ref{thm:main}.
\thmmain*

By Proposition~\ref{prop:decompose_localize}, we can write $\mZ(f)$ as a finite intersection $\bigcap_{i = 1}^d \mZ(f_i)$, where each $f_i$ is a linear-exponential map to a module $M_i$ over a local ring $R_i$ with residue field of characteristic $p_i \mid \ell$.
We obtain Theorem~\ref{thm:main} if we show that each $\mZ(f_i)$ is $p_i$-normal.
Therefore Theorem~\ref{thm:main} reduces to the following special case where $R$ is local:

\begin{theorem}[Linear-exponential equation in modules over local rings]\label{thm:local}
    Let $R$ be a local ring with maximal ideal $\frm$ and residue field of characteristic $p$.
    Let $M$ be a finitely generated $R$-module with $\frm^tM=0$ for some $t \in \N$, and let $f(\ba)=\sum_{i = 1}^k\br_i^{\ba} m_i$ where $\br_1,\ldots,\br_k\in (R^{\times})^n, m_1,\ldots,m_k\in M$.
    Then the zero set $\mZ(f)$ is effectively $p$-normal.
\end{theorem}

To prove Theorem~\ref{thm:local}, we will reduce it to the theorem of Derksen-Masser on the case of fields:

{\renewcommand\footnote[1]{}\thmfield*}

In order to reduce Theorem~\ref{thm:local} to Derksen-Masser, our overall idea is to express a linear-exponential equation over the $R$-module $M$ to an equivalent system of linear-exponential equations over the residue field $R/\frm$.
First we use an example to illustrate this idea and the main difficulties:

\begin{example}[Overview of main idea]\label{example:modfour}
    Consider the following linear-exponential map $f \colon \Z \rightarrow \Z_{/4}(X)$, where $\Z_{/4}(X)$ is considered as a module over itself:
    \begin{equation}\label{eq:example_modfour}
        f(a) \coloneqq (X+1)^a - X^a - 1^a \cdot (2X+1).
    \end{equation}
    For convenience, we write ``$\bmod\, 2$'' for ``$\bmod\, \gen{2}$'' in $\Z_{/4}(X)$.
    First we consider the solution set of the modular equation
    $
    f(a) \equiv 0 \pmod 2
    $.
    This can be written as the zero set of the map $(f \bmod 2) \colon \Z \rightarrow \F_2(X)$, where 
    \[
    (f \bmod 2)(a) = (X+1)^a - X^a - 1^a.
    \]
    In particular, we have $\mZ(f) \subseteq \mZ(f \bmod 2)$, however in general $\mZ(f) \neq \mZ(f \bmod 2)$. 
    By Derksen-Masser, the zero set $\mZ(f \bmod 2)$ is effectively $2$-normal; our hope is to write $\mZ(f)$ as an intersection of $\mZ(f \bmod 2)$ with another $2$-normal set.

    Note that whenever $f(a) \equiv 0 \pmod{2}$, we can write $f(a) = 2 \widehat{f(a)}$ for some $\widehat{f(a)} \in \Z_{/4}(X)$; however the value of $\widehat{f(a)}$ is not defined for the remaining $a$ where $f(a) \not\equiv 0 \pmod{2}$.
    Our aim is to ``complete'' $\widehat{f(a)}$ into a well-defined map over all $a \in \Z$.
    That is, we want to construct a map $g \colon \Z \rightarrow \Z_{/4}(X)$, such that for every $a$ satisfying $f(a) \equiv 0 \pmod{2}$, we have $f(a) = 2g(a)$.
    If furthermore we can construct $g$ to be linear-exponential, then we have
    \begin{align*}
    f(a) = 0 &\; \iff f(a) \equiv 0 \pmod 2, \quad 2g(a) = 0 \\
    &\; \iff f(a) \equiv 0 \pmod 2, \quad g(a) \equiv 0 \pmod 2.
    \end{align*}
    since we work in characteristic $4$.
    Therefore $\mZ(f) = \mZ(f \bmod 2) \cap \mZ(g \bmod 2)$.
    Since $(f \bmod 2)$ and $(g \bmod 2)$ are linear-exponential maps over the residue field $\Z_{/4}(X)/\gen{2} = \F_2(X)$, the $2$-normality of $\mZ(f)$ would follow immediately from Theorem~\ref{thm:DM} as a blackbox.
    \hfill $\blacksquare$
\end{example}

Example~\ref{example:modfour} shows that the reduction from local rings to fields boils down to constructing such a linear-exponential map $g$.
Formally speaking, let $R$ be a local ring with maximal ideal $\frm$.
Given a linear-exponential map $f(\ba)=\sum_{i = 1}^k \br_i^{\ba} m_i$ over an $R$-module $M$, we hope to find $\br'_1, \ldots, \br'_{k'}\in (R^{\times})^n, m'_1, \ldots, m'_{k'}\in \frm M$, such that 
\[
f(\ba) \in \frm M \implies f(\ba) = \sum_{i = 1}^{k'}(\br_i')^{\ba} m'_i.
\]
In other words, we hope to find another representation for $f$ which is ``internal'' to $\frm M$ whenever $f(\ba) \in \frm M$.
Such an internal representation would allow us to rewrite $f(\ba) = 0$ as a system of two equations 
\[
f(\ba) \equiv 0 \pmod{\frm M}, \quad \text{and} \quad \sum_{i = 1}^{k'}(\br_j')^{\ba} m'_j = 0.
\]
This reduces an equation over the module $M$ to a system of two equations, one over $M/\frm M$ and one over $\frm M$, allowing for a proof by induction on the structure of $M$.

Constructing an internal representation is difficult, even for the map $f$ in Example~\ref{example:modfour}.
A more striking example is the map $f(a) = (X+1)^a - (X-1)^a$ in $\Z_{/4}(X)$.
Although $f(a) \equiv 0 \pmod 2$ for all $a \in \Z$, it's not clear how to write down a linear-exponential map $g$ such that $f(a) = 2g(a)$ for all $a$.
In fact, there are explicit examples where an internal representation provably does not exist (see Remark~\ref{rmk:counterexample}).

Nevertheless, we now show that an internal representation can indeed be obtained if the entries of $\br_i$ are $p^{t-1}$-th powers in $R^{\times}$.
This would allow us to separate $f$ into disjoint cases according to the residues of $\ba$ modulo $p^{t-1}$, and obtain an internal representation for each residue.
\begin{proposition}[internal representation]\label{prop:internal_rep}
    Let $R$ be a local ring with maximal ideal $\frm$ and residue field of characteristic $p$.
    Let $M$ be a finitely generated $R$-module with $\frm^tM=0$ for $t \in \N$.
    Let $f:\Z^n\rightarrow M$ be a linear-exponential map defined by $f(\ba)=\sum_{i=1}^k \br_i^{p^{t-1}\ba}m_i$, where $\br_i\in (R^{\times})^n,m_i\in M$. 
    Then there exists an effectively computable linear-exponential map $\tilde f:\Z^n\rightarrow \frm M$, internal to the $R$-module $\frm M$, such that $\tilde f(\ba)=f(\ba)$ whenever $f(\ba)\in \frm M$. 
\end{proposition}
\begin{proof}
    An elementary illustration of the proof will be given in Example~\ref{example:continued}.
    We write $q \coloneqq p^{t-1}$.
    The proof proceeds in two steps.
    In the first step, we will find $d \leq k$ and homogeneous linear polynomials $\phi_{d+1}, \ldots, \phi_k \in R[X_1, \ldots, X_d]$, such that (possibly after rearrangement of indices):
    \begin{equation*}
    f(\ba)=\sum_{i=1}^k \br_i^{q\ba}m_i \in \frm M \implies \br_{d+1}^{q\ba} = \phi_{d+1}(\br_1^{\ba}, \ldots, \br_d^{\ba})^q, \ldots, \br_{k}^{q\ba} = \phi_k(\br_1^{\ba}, \ldots, \br_d^{\ba})^q.
    \end{equation*}
    In the second step, in the linear-exponential map $f(\ba)=\sum_{i=1}^k \br_i^{q\ba}m_i$, we substitute the terms $\br_{d+1}^{q\ba}, \ldots, \br_{k}^{q\ba}$, using $\phi_{d+1}(\br_1^{\ba}, \ldots, \br_d^{\ba})^q, \ldots, \phi_k(\br_1^{\ba}, \ldots, \br_d^{\ba})^q$, and show that the resulting expression $\tilde f(\ba)$ is a linear-exponential map internal to $\frm M$.
    This will yield $f(\ba)\in \frm M \implies \tilde f(\ba)=f(\ba)$.
    The following are the details of the two steps.
    \smallskip
    
    \textbf{Step 1.}
    Let $V=M/\frm M$.
    Since $M$ is finitely generated as an $R$-module, $V$ is a finite dimensional vector space over the residue field $F \coloneqq R/\frm$. 
    Let $\overline{\cdot} \colon M\rightarrow V,\; \overline\cdot \colon R\rightarrow F$ denote the projections modulo $\frm$.
    Let $h:F^k\rightarrow V$ be the map given by $h(x_1, \ldots ,x_k)=\sum_{i = 1}^k x_i^q\overline{m_i}$, then
    \[
    \sum_{i=1}^k \br_i^{q\ba}m_i \in \frm M \iff h\left(\overline{\br_1^{\ba}}, \ldots, \overline{\br_k^{\ba}}\right) = 0_V.
    \]
    Consider the set of points on which $h$ vanishes:
    \begin{equation}\label{eq:def_space_H}
    H \coloneqq \{(x_1, \ldots, x_k) \mid h(x_1, \ldots, x_k) = 0_V\} \subseteq F^k.
    \end{equation}
    Then $\sum_{i=1}^k \br_i^{q\ba}m_i \in \frm M \iff \left(\overline{\br_1^{\ba}}, \ldots, \overline{\br_k^{\ba}}\right) \in H$.
    The key observation is the following:

    \medskip
    \begin{adjustwidth}{3mm}{3mm}
    \textbf{Claim.} $H$ is an $F$-linear subspace of $F^k$.
    \medskip
    
    \noindent$\blacktriangleright$
    Let $(x_1, \ldots ,x_k), (y_1, \ldots ,y_k) \in H$ and $c \in F$, then (since $F$ has characteristic $p$):
    \[
    h(x_1 + y_1, \ldots ,x_k + y_k) = \sum_{i=1}^k (x_i+y_i)^q \overline{m_i} = \sum_{i=1}^k x_i^q \overline{m_i} + \sum_{i=1}^k y_i^q \overline{m_i} = 0_V,
    \]
    and
    \[
    h(cx_1, \ldots , c x_k) = \sum_{i=1}^k (cx_i)^q \overline{m_i} = c^q \sum_{i=1}^k x_i^q \overline{m_i} = 0_V.
    \]
    Therefore $\bx, \by \in H \implies \bx+\by \in H$, and $\bx \in H, c \in F \implies c \bx \in H$, so $H$ is an $F$-linear space.
    \hfill $\blacktriangleleft$
    \end{adjustwidth}
    \medskip

    \noindent In fact, an $F$-basis of $H$ can be effectively computed, but we delay this proof to Lemma~\ref{lem:effective_H}.
    Coming back to the proposition, write $\dim_{F} H=d$. 
    After relabelling the coordinates if necessary, we may assume that the first $d$ coordinate functions of $H$ are linearly independent. 
    It follows that every vector in $H$ is uniquely determined by its first $d$ coordinates, and hence
    \[
    \textstyle
    H = \left\{\left(x_1, \ldots, x_d, \sum_{i = 1}^d c_{i, d+1} x_i, \ldots, \sum_{i = 1}^d c_{i, k} x_i \right) \mid x_1, \ldots, x_d \in F\right\}
    \]
    for suitable $c_{i, j} \in F,\; 1 \leq i \leq d,\; d+1 \leq j \leq k$.

    Recall that $\sum_{i=1}^k \br_i^{q\ba}m_i \in \frm M$ if and only if $\left(\overline{\br_1^{\ba}}, \ldots, \overline{\br_k^{\ba}}\right) \in H$. Therefore
    \begin{equation*}
    \sum_{i=1}^k \br_i^{q\ba}m_i \in \frm M \iff \overline{\br_j^{\ba}}=\sum_{i = 1}^d c_{i,j} \overline{\br_i^{\ba}}, \; \text{ for } j = d+1, \ldots, k. 
    \end{equation*}
    Recall that $r \equiv s \pmod{\frm} \implies r^{p^{t-1}} \equiv s^{p^{t-1}} \pmod{\frm^t}$ (Lemma~\ref{lem:lifting}).
    Since $\frm^t M = 0$, we can without loss of generality suppose $\frm^t = 0$ by replacing $R$ with $R/\frm^t$.
    For each $i, j$, fix an $s_{i,j} \in R$ such that $\overline{s_{i,j}}=c_{i,j}$, then
    \begin{equation}\label{eq:rep_rqa}
    \overline{\br_j^{\ba}}=\sum_{i = 1}^d c_{i,j} \overline{\br_i^{\ba}} \;\iff\; \br_j^{\ba} \equiv \sum_{i = 1}^d s_{i,j}\br_i^{\ba} \pmod{\frm} \;\implies\; \br_{j}^{q\ba} = \left( \sum_{i = 1}^d s_{i,j}\br_i^{\ba} \right)^q,
    \end{equation}
    since $q = p^{t-1}$ and $\frm^t = 0$.
    The last equation in~\eqref{eq:rep_rqa} expresses each $\br_{j}^{q\ba}, j = d+1, \ldots, k$, in terms of $\br_{1}^{\ba}, \ldots, \br_{d}^{\ba}$, valid whenever $\sum_{i=1}^k \br_i^{q\ba}m_i \in \frm M$.
    \smallskip

    \textbf{Step 2.}
    Suppose now $f(\ba)=\sum_{i=1}^k \br_i^{q\ba}m_i \in \frm M$.
    In $f(\ba)=\sum_{i=1}^k\br_i^{q\ba}m_i$, we substitute $\br_{d+1}^{q\ba}, \ldots, \br_{k}^{q\ba},$ using the last equation in~\eqref{eq:rep_rqa} and expand.
    Note that every coefficient of the polynomial $(\sum_{i=1}^d X_i)^q - \sum_{i=1}^d X_i^q$ is divisible by $p$, so
    \begin{align*}
        f(\ba) =\sum_{i=1}^k\br_i^{q\ba}m_i & = \sum_{i=1}^d\br_i^{q\ba}m_i + \sum_{j=d+1}^k \left( \sum_{i = 1}^d s_{i,j}\br_i^{\ba} \right)^q m_{j}\\
        & = \sum_{i=1}^d\br_i^{q\ba}m_i + \sum_{j=d+1}^k \left(\sum_{i = 1}^d s_{i,j}^q\br_i^{q\ba} + p(\cdots)\right)m_{j} \\
        & = \sum_{i=1}^d \br_i^{q\ba} \left(m_i + \sum_{j = d+1}^k s_{i, j}^q m_j\right) + p(\cdots ),
    \end{align*}
    where inside the parenthesis $(\cdots)$ is a linear-exponential form 
    \[
    (\cdots) = \sum_{n_1 + \cdots + n_d = q} \br_1^{n_1\ba} \cdots \br_d^{n_d\ba} \cdot t_{n_1, \ldots, n_d} = \sum_{n_1 + \cdots + n_d = q} (\br_1^{n_1} \cdots \br_d^{n_d})^{\ba} \cdot t_{n_1, \ldots, n_d},
    \]
    with $t_{n_1, \ldots, n_d} \in M$, and $\br_1^{n_1} \cdots \br_d^{n_d} = (r_{11}^{n_1} \cdots r_{d1}^{n_d}, \ldots, r_{1n}^{n_1} \cdots r_{dn}^{n_d})$ $\in (R^{\times})^n$.
    Therefore $p(\cdots)$ is a linear-exponential form internal to $pM\subseteq\frm M$. 
    For the remaining coefficients, we have
    \[
    \overline{\left(m_i + \sum_{j = d+1}^k s_{i, j}^q m_j\right)} = \overline{m_i}+\sum_{j = d+1}^k c_{i,j}^q\overline{m_j} = h(0,\ldots,0,\underset{i\text{-}th}{\underset{\uparrow}{1}},0,\ldots,0,c_{i,d+1},\ldots,c_{i,k}) = 0_V,
    \]
    for $i = 1, \ldots, d$.
    Therefore $m_i + \sum_{j = d+1}^k s_{i, j}^q m_j \in \frm M$.
    It follows that the linear-exponential map
    \[
    \tilde{f}(\ba) \coloneqq \sum_{i=1}^d \br_i^{q\ba} \left(m_i + \sum_{j = d+1}^k s_{i, j}^q m_j\right) + p(\cdots )
    \]
    is internal to $\frm M$, and is equal to $f(\ba)$ whenever $f(\ba) \in \frm M$.
\end{proof}

To complete the effectiveness part of Proposition~\ref{prop:internal_rep}, it remains to compute $H$:
\begin{lemma}\label{lem:effective_H}
    An $F$-basis of $H$ can be effectively computed.
\end{lemma}
\begin{proof}
    Recall that $H$ is defined as the set of $(x_1, \ldots, x_k) \in F^k$ where $h(x_1, \ldots ,x_k)=\sum_{i = 1}^k x_i^q\overline{m_i}$ vanishes.
    Here, $\overline{m_1}, \ldots, \overline{m_k} \in V$.
    Let $u \coloneqq \dim_F V$ and fix a basis of $V$ over $F$. For each $i = 1, \ldots, k$, write $\overline{m_i}$ as $(m_{i1}, \ldots, m_{iu}) \in F^u$ over this basis.
    Let $F'$ be the field extension of $F$ obtained by adjoining the $q$-th roots of the elements $m_{ij}, i = 1, \ldots, k; j =1, \ldots, u$.
    Then the extension $F'/F$ is finite; in particular, $F'$ is a finite dimensional $F$-linear space (see~\cite[Chapter~3.5]{roman2006field}).
    Then (since $q = p^{t-1}$ and $F'$ has characteristic $p$):
    \begin{multline}\label{eq:linear_over_F_root}
        (x_1, \ldots, x_k) \in H \iff \sum_{i = 1}^k x_i^q\overline{m_i} = 0_V \iff \sum_{i = 1}^k x_i^q m_{i1} = \cdots = \sum_{i = 1}^k x_i^q m_{iu} = 0_F \\
        \iff \sum_{i = 1}^k x_i m_{i1}^{1/q} = \cdots = \sum_{i = 1}^k x_i m_{iu}^{1/q} = 0_{F'}.
    \end{multline}
    Let $H'$ be the set of $(x_1, \ldots, x_k) \in \left(F'\right)^k$ that satisfy the system of linear equations in~\eqref{eq:linear_over_F_root}.
    Then $H'$ is an $F'$-linear space and $H = H' \cap F^k$.
    In particular, both $H'$ and $F^k$ are $F$-linear subspaces of $\left(F'\right)^k$ and effectively computable.
    Since $\dim_F \left(F'\right)^k$ is finite, an $F$-basis of $H' \cap F^k$ can be computed by linear algebra over $F$.
\end{proof}

Proposition~\ref{prop:internal_rep} provides the internal representation required for the proof of Theorem~\ref{thm:local}: 
\begin{proof}[Proof of Theorem~\ref{thm:local}]
    Recall that our goal is to prove that the zero set $\mZ(f)$ is effectively $p$-normal, where $f(\ba) = \sum_{i = 1}^k \br_i^{\ba} m_i$ over a linear-exponential map in the finitely generated $R$-module $M$ satisfying $\frm^t M = 0$.
    We proceed by induction on $t$.
    
    If $t=1$, then $\frm M = 0$, and hence $M=M/\frm M$. 
    Viewing $M/\frm M$ as an $R/\frm$-module, we see that $M$ is in fact a finite-dimensional vector space over the residue field $F=R/\frm$.
    Fix a basis of $M$ over $F$ and identify $M$ with $F^d$ for some $d\in\mathbb N$.
    Then $\sum_{i = 1}^k \br_i^{\ba} m_i = 0_M$ if and only if $\sum_{i = 1}^k \br_i^{\ba} m_{i1} = \cdots = \sum_{i = 1}^k \br_i^{\ba} m_{id} = 0_F$, where $(m_{i1},\ldots,m_{id})\in F^d$ denotes the coordinate vector of $m_i$, and the entries of $\br_i$ are considered as elements in the residue field $R/\frm = F$.
    Therefore we are reduced to the case of fields by writing $\mZ(f) = \mZ\big(\sum_{i = 1}^k \br_i^{\ba} m_{i1}\big) \cap \cdots \cap \mZ\big(\sum_{i = 1}^k \br_i^{\ba} m_{id}\big)$.
    By Derksen-Masser (Theorem~\ref{thm:DM}), the zero sets $\mZ(\sum_{i = 1}^k \br_i^{\ba} m_{ij}), j = 1, \ldots, d$, are effectively $p$-normal.
    Therefore $\mZ(f)$ is also effectively $p$-normal by Lemma~\ref{lem:inter_normal}.

    For the induction step, suppose $t \geq 2$ and that effective $p$-normality has been proven in all finitely generated modules $M$ with $\frm^{t-1}M = 0$.
    Consider a finitely generated module $M$ with $\frm^tM = 0$.
    Let $q = p^{t-1}$.
    We can write $\mZ(f)$ as a disjoint union $\bigcup_{\bb \in \{0, 1, \ldots, q-1\}^k} q \cdot \mZ(f_{\bb}) + \bb$, where $f_{\bb}(\ba) \coloneqq f(q\ba+\bb) = \sum_{i = 1}^k \br_i^{q\ba} \cdot (\br_i^{\bb}m_i)$.
    Since $p$-normal sets are closed under finite union and affine transformations, it suffices to show that any linear-exponential map of the form $f(\ba) = \sum_{i = 1}^k \br_i^{q\ba} m_i$ has an effectively $p$-normal zero set.

    By Proposition~\ref{prop:internal_rep}, we can compute a linear-exponential map $\tilde f(\ba) = \sum_{i = 1}^{k'} (\br'_i)^{q\ba} m'_i$ where $\br'_1, \ldots, \br'_{k'}\in (R^{\times})^n, m'_1, \ldots, m'_{k'}\in \frm M$, such that $\tilde f(\ba)=f(\ba)$ whenever $f(\ba) \in \frm M$.
    Thus the equation $f(\ba) = 0$ in $M$ is equivalent to the system of two equations
    \[
    f(\ba) \equiv 0 \pmod{\frm M}, \quad \tilde f(\ba) = 0 \text{ (in $\frm M$).}
    \]
    Therefore $\mZ(f) = \mZ(f \bmod \frm M) \cap \mZ(\tilde f)$.
    Here, $f \bmod \frm M$ and $\tilde f$ are linear-exponential maps over $M/\frm M$ and $\frm M$, respectively.
    Note that $\frm (M/\frm M) = 0$ and $\frm^{t-1}(\frm M) = 0$.
    Apply the induction hypothesis on $M/\frm M$ and on the submodule of $\frm M$ generated by the $m'_i$'s.
    We obtain that both zero sets are effectively $p$-normal.
    Hence $\mZ(f)$ is effectively $p$-normal.
\end{proof}

We now illustrate the proof of Proposition~\ref{prop:internal_rep} and Theorem~\ref{thm:local} by explicitly computing the internal representation and the zero sets in Example~\ref{example:modfour}.

\begin{example}[continuation of Example~\ref{example:modfour}]\label{example:continued}
    We compute the zero set of the following map over $\Z_{/4}(X)$:
    \begin{equation}\label{eq:example_modfour_bis}
        f(a) \coloneqq (X+1)^a - X^a - 1^a \cdot (2X+1).
    \end{equation}
    We split into two cases according to the residue of $a$ modulo $2$.
    \medskip

    \textbf{When $a$ is even.} Write $a = 2b$, then expression~\eqref{eq:example_modfour_bis} becomes
    \[
    f(b) = (X+1)^{2b} - X^{2b} - 1^{2b} \cdot (2X+1),
    \]
    so
    \[
    (f \bmod 2)(b) = (X+1)^{2b} - X^{2b} - 1^{2b} \in \F_2(X).
    \]
    Define the map $h \colon \F_2(X)^3 \rightarrow \F_2(X)$,
    \[
    h(x_1, x_2, x_3) = x_1^2 - x_2^2 - x_3^2.
    \]
    This allows us to write
    \[
    (f \bmod 2)(b) = h\big((X+1)^b, X^b, 1^b\big).
    \]
    Since $x_1^2 - x_2^2 - x_3^2 = x_1^2 + x_2^2 + x_3^2 = (x_1 + x_2 + x_3)^2 = (x_1 + x_2 - x_3)^2$ in characteristic $2$, we have
    \begin{align*}
    H \coloneqq h^{-1}(0) & = \left\{(x_1, x_2, x_3) \in \F_2(X)^3 \;\middle|\; x_1^2 - x_2^2 - x_3^2 = 0 \right\}. \\
    & = \{(x_1, x_2, x_1 + x_2) \mid x_1, x_2 \in \F_2(X)\}.
    \end{align*}
    Therefore
    \[
    f(b) \equiv 0 \pmod{2} \; \iff \big((X+1)^b, X^b, 1^b\big) \in H \iff  1^b \equiv (X+1)^b + X^b \pmod{2}.
    \]
    By Lemma~\ref{lem:lifting}, we have
    \begin{equation}\label{eq:replace}
    1^b \equiv (X+1)^b + X^b \pmod 2 \; \implies 1^{2b} = \big((X+1)^b + X^b\big)^2.
    \end{equation}
    This means that whenever $f(b) \equiv 0 \pmod 2$, we can replace $1^{2b}$ by $\big((X+1)^b + X^b\big)^2$.
    Hence
    \begin{align*}
    f(b) \equiv 0 \pmod 2 \implies f(b) & = (X+1)^{2b} - X^{2b} - 1^{2b} \cdot (2X+1) \\
    & = (X+1)^{2b} - X^{2b} - \big((X+1)^b + X^b\big)^2 \cdot (2X+1) \\
    & = - (X+1)^{2b} \cdot 2X - X^{2b} \cdot (2X + 2) - (X^2+X)^b \cdot 2.
    \end{align*}
    Thus whenever $f(b) \equiv 0 \pmod 2$, we have an internal representation $f = 2g$ with
    \[
    g(b) \coloneqq - (X+1)^{2b} \cdot X - X^{2b} \cdot (X + 1) - (X^2+X)^b.
    \]
    Therefore $\mZ(f) = \mZ(f \bmod 2) \cap \mZ(g \bmod 2)$, where both $\mZ(f \bmod 2)$ and $\mZ(g \bmod 2)$ are effectively $2$-normal since they are linear-exponential maps over the residue field $\F_2(X)$ (Theorem~\ref{thm:DM}).

    We can explicitly compute these zero sets:
    by Example~\ref{example:Derksen} we have $\mZ(f \bmod 2) = \{2^k \mid k \in \N\}$, and by direct computation we have $\mZ(g \bmod 2) = \{0, 1\}$.
    Their intersection is $\{1\}$.
    Therefore, the only solution to $f = 0$ with even $a$ is $a = 2b = 2$.
    \medskip

    \textbf{When $a$ is odd.} The strategy is the same as the previous case.
    Write $a = 2b+1$, then expression~\eqref{eq:example_modfour_bis} becomes
    \[
    f(b) = (X+1)^{2b} \cdot (X+1) - X^{2b} \cdot X - 1^{2b} \cdot (2X+1).
    \]
    Define the map $h \colon \F_2(X)^3 \rightarrow \F_2(X)$, 
    \[
    h(x_1, x_2, x_3) = x_1^2 \cdot (X+1) + x_2^2 \cdot X + x_3^2.
    \]
    This allows us to write
    $
    (f \bmod 2)(b) = h\big((X+1)^b, X^b, 1^b\big)
    $.
    This time, we have
    \begin{align*}
    H = h^{-1}(0)
    & = \left\{(x_1, x_2, x_3) \in \F_2(X)^3 \;\middle|\; \left(x_1 \cdot (\sqrt{X} + 1) + x_2 \cdot \sqrt{X} + x_3\right)^2 = 0 \right\} \\
    & = \left\{(x_1, x_2, x_3) \in \F_2(\sqrt{X})^3 \;\middle|\; x_1(\sqrt{X} + 1) + x_2 \sqrt{X} + x_3 = 0 \right\} \cap \F_2(X)^3 \\
    & = \left\{(x_1, x_2, x_3) \in \F_2(X)^3 \;\middle|\; x_1\sqrt{X} + x_2 \sqrt{X} = 0, x_1 + x_3 = 0 \right\} \\
    & = \left\{(x_1, x_2, x_3) \in \F_2(X)^3 \;\middle|\; x_1 + x_2 = 0, x_1 + x_3 = 0 \right\} \\
    & = \left\{(x_1, x_1, x_1) \;\middle|\; x_1 \in \F_2(X)\right\}.
    \end{align*}
    Therefore
    \[
    f(b) \equiv 0 \pmod 2 \; \iff \big((X+1)^b, X^b, 1^b\big) \in H \iff (X+1)^b \equiv X^b \equiv 1^b \pmod 2.
    \]
    By Lemma~\ref{lem:lifting}, we have
    \[
    (X+1)^b \equiv X^b \equiv 1^b \pmod 2 \; \implies (X+1)^{2b} = X^{2b} = 1^{2b}.
    \]
    This means that whenever $f(b) \equiv 0 \pmod 2$, we can replace both $X^{2b}$ and $1^{2b}$ by $(X+1)^{2b}$.
    Hence
    \begin{samepage}
    \begin{align*}
    f(b) \equiv 0 \pmod 2 \; \implies f(b) & = (X+1)^{2b} \cdot (X+1) - (X+1)^{2b} \cdot X - (X+1)^{2b} \cdot (2X+1) \\
    & = (X+1)^{2b} \cdot (-2X).
    \end{align*}
    \end{samepage}
    Thus whenever $f(b) \equiv 0 \pmod 2$, we have an internal representation $f = 2g$ with
    \[
    g(b) = (X+1)^{2b} \cdot (-X).
    \]
    Therefore $\mZ(f) = \mZ(f \bmod 2) \cap \mZ(g \bmod 2)$, reducing to the case over the residue field $\F_2(X)$.
    Explicitly, $\mZ(f \bmod 2) = \{0\}$, $\mZ(g \bmod 2) = \emptyset$, so there are no solutions to $f = 0$ with odd $a$.
    \hfill $\blacksquare$
\end{example}

Finally, the general case of Theorem~\ref{thm:main} follows immediately from the local case (Theorem~\ref{thm:local}):

\begin{proof}[Proof of Theorem~\ref{thm:main}]
    Let $f \colon \Z^n \rightarrow M$ be a linear-exponential map where $\ell M = 0$.
    By Proposition~\ref{prop:decompose_localize} we can write $\mZ(f) = \bigcap_{i = 1}^d \mZ(f_i)$, and by Theorem~\ref{thm:local} each $\mZ(f_i)$ is effectively $p_i$-normal for some prime $p_i \mid \ell$.
    Grouping up the $p$-normal sets for the same $p$ and replacing their intersection with a single $p$-normal set, we obtain $\mZ(f) = \bigcap_{p \mid \ell} S_p$ where each $S_p$ is effectively $p$-normal.
\end{proof}

\begin{remark}\label{rmk:counterexample}
    In the statement of Proposition~\ref{prop:internal_rep}, it is essential that the map $f(\ba)=\sum_{i=1}^k \br_i^{p^{t-1}\ba}m_i$ is defined using the $p^{t-1}$-th powers of $\br_i$.
    Below, we give two examples of a linear-exponential map $f(\ba)=\sum_{i=1}^k \br_i^{\ba}m_i$ where an internal representation in $\frm M$ does not exist, despite $f$ taking values only in $\frm M$.

    \begin{enumerate}[nosep, label=(\arabic*)]
        \item Let $R = \F_2[X]/\gen{X^2}$, then $R$ is a local ring of four elements, with maximal ideal $\frm=XR$.
        Take $M = R$ and define the linear-exponential map $f \colon \Z \rightarrow M$, 
        \[
        f(a) = (X+1)^a- 1^a.
        \]
        Then $f(a) = aX \in \frm M$ for all $a \in \Z$.
        Let $\overline{\cdot} \colon R \rightarrow R/\frm = \F_2$ denote the projection.
        Suppose we can represent $f(a)$ internally in $\frm M$ as $\sum_{i = 1}^k r_i^a (Xm_i)$ with $r_i \in R^{\times} = R \setminus \frm$.
        Then $a \equiv \sum_{i = 1}^k \overline{r_i}^a \overline{m_i} \pmod{2}$ for all $a\in\Z$. 
        But $r_i \in R^{\times}$ means $\overline{r_i} \neq 0$, so $\overline{r_i}^a  = 1^a = 1$ is constant.
        Then $\sum_{i = 1}^k \overline{r_i}^a \overline{m_i}$ is constant, a contradiction.
        \item Let $R = \Z_{/p^2}$ where $p$ is a prime number, then $R$ is a local ring with maximal ideal $\frm=pR$.
        Take $M = R$ and define the linear-exponential map $f \colon \Z \rightarrow M$, 
        \[
        f(a) = (p+1)^a- 1^a.
        \]
        Then $f(a) = ap \in \frm M$ for all $a \in \Z$.
        Suppose we can represent $f(a)$ internally in $\frm M$ as $\sum_{i = 1}^k r_i^a (pm_i)$ with $r_i \in R^{\times} = R \setminus pR$.
        Then $a \equiv \sum_{i = 1}^k r_i^a m_i \pmod{p}$ for all $a$.
        Replacing $a$ with $(p-1)a$, we get 
        \[
        (p-1)a \equiv \sum_{i = 1}^k r_i^{(p-1)a} m_i \equiv \sum_{i = 1}^k 1^{a} m_i = \sum_{i = 1}^k m_i \pmod{p}
        \]
        by Fermat's little theorem.
        This is a contradiction since the left hand side depends on $a$, while the right hand side is constant.
    \end{enumerate}
    \smallskip
    Note that in either example, replacing $a$ with $pa$ makes the map $f$ a constant in $\frm M$, which results in an internal representation.
\end{remark}

\section{Multi-dimensional Skolem-Mahler-Lech in positive characteristic}\label{sec:multi_dim}

\subsection{Multi-dimensional recurrence sequences and formal power series}\label{subsec:def_LRS}

Let $M$ be a module over a ring $R$.
Recall that we define an \emph{$n$-dimensional linear recurrence sequence} to be a map $\gamma \colon \N^n \rightarrow M$, for which there exists an integer $d \geq 1$ and coefficients $c_{ij} \in R, 1 \leq i \leq n, 1 \leq j \leq d$, such that the following recurrence holds for $i = 1, \ldots, n$:
\begin{align}\label{eq:recurrence_def}
\gamma(a_1, \ldots, a_i, \ldots, a_n) = c_{i1} \cdot \gamma(a_1, \ldots, a_i-1, \ldots, a_n) + \cdots + c_{id} \cdot \gamma(a_1, \ldots, a_i-d, \ldots, a_n), \quad & \\
\text{ for all $(a_1, \ldots, a_n) \in \N^n$ with $a_i\geq d$}.& \nonumber
\end{align}
In other words, we require that $\gamma(a_1, \ldots, a_n)$ satisfies a linear recurrence in the dimensions where the index is at least $d$.
Thus, $\gamma$ is completely determined by its \emph{initial values}
\[
\gamma(a_1, \ldots, a_n) \in M, \quad (a_1, \ldots, a_n) \in \{0, 1, \ldots, d-1\}^n.
\]
Note that this definition includes the case where the order $d$ is different in each dimension, as we can simply take $d$ to be their maximum.

Let $R[[T_1,\ldots,T_n]]$ denote the ring of formal power series in the variables $T_1, \ldots, T_n$ over $R$.
It consists of the formal sums
\[
\sum_{(a_1, \ldots, a_n) \in \N^n} r_{a_1, \ldots, a_n} T_1^{a_1} \cdots T_n^{a_n},
\]
where $r_{a_1, \ldots, a_n} \in R$.
Similarly, let $M[[T_1,\ldots,T_n]]$ denote the set of formal sums
\[
\sum_{(a_1, \ldots, a_n) \in \N^n} m_{a_1, \ldots, a_n} T_1^{a_1} \cdots T_n^{a_n},
\]
where $m_{a_1, \ldots, a_n} \in M$. 
Then $M[[T_1,\ldots,T_n]]$ is naturally an $R[[T_1,\ldots,T_n]]$-module.

Given a sequence $\gamma \colon \N^n \rightarrow M$, we can define the power series associated to $\gamma$:
\[
\phi_{\gamma}(T_1, \ldots, T_n) \coloneqq \sum_{(a_1, \ldots, a_n) \in \N^n} \gamma(a_1, \ldots, a_n) T_1^{a_1} \cdots T_n^{a_n} \in M[[T_1, \ldots, T_n]].
\]
Given the recurrence relations~\eqref{eq:recurrence_def} for $i = 1, \ldots, n$, we define their associated \emph{recurrence polynomials} (or \emph{characteristic polynomials}) to be
\[
g_i(T) \coloneqq T^{d} - c_{i1} T^{d-1} - c_{i2} T^{d-2} - \cdots - c_{id} \in R[T], \quad i = 1, \ldots, n.
\]
For a univariate polynomial $g$, define its \emph{reciprocal} to be $g^\circ(T) \coloneqq T^{\deg(g)}g(T^{-1})$.
Then
\begin{equation}\label{eq:gi_circ}
g_i^\circ(T) = 1 - c_{i1} T - c_{i2} T^{2} - \cdots - T^{d} c_{id}.
\end{equation}
It is easy to see that every $g \in R[[T]]$ with constant coefficient $1$ is invertible.
In particular, the inverses $\frac{1}{g_i^\circ(T)}$ are well-defined elements in $R[[T]]$.

\begin{observation}
    A sequence $\gamma \colon \N^n \rightarrow M$ satisfies the recurrence relations~\eqref{eq:recurrence_def} for $i = 1, \ldots, n$, if and only if
    \begin{equation}\label{eq:phi_rat}
    \phi_{\gamma}(T_1, \ldots, T_n) = \frac{h(T_1, \ldots, T_n)}{\prod_{i=1}^n g_i^\circ(T_i)}
    \end{equation}
    for some finitely supported $h(T_1,\ldots,T_n)\in M[[T_1,\ldots,T_n]]$ with degree at most $d-1$ in each variable.
\end{observation}
\begin{proof}
Fix $i\in\{1,\ldots,n\}$. The coefficient of $T_1^{a_1}\cdots T_n^{a_n}$ in $g_i^\circ(T_i)\phi_\gamma(T_1,\ldots,T_n)$, where $a_i\ge d$, is
\[
    \gamma(a_1, \ldots, a_i, \ldots, a_n) - c_{i1} \cdot \gamma(a_1, \ldots, a_i-1, \ldots, a_n) - \cdots - c_{id} \cdot \gamma(a_1, \ldots, a_i-d, \ldots, a_n).
\]
Hence, $\gamma$ satisfies the recurrence relation~\eqref{eq:recurrence_def} in the $i$-th dimension if and only if $g_i^\circ(T_i)\phi_\gamma(T_1,\ldots,T_n)$ contains no monomials whose degree in $T_i$ is at least $d$.

Since multiplication by $g_i^\circ(T_i)$ affects only the variable $T_i$, it leaves the degrees in the remaining variables unchanged. Therefore, $\gamma$ satisfies~\eqref{eq:recurrence_def} in every dimension if and only if
\[
\textstyle
h(T_1,\ldots,T_n)
\coloneqq
\prod_{i=1}^n g_i^\circ(T_i)\,
\phi_\gamma(T_1,\ldots,T_n)
\]
has degree at most $d-1$ in each variable.
\end{proof}

From the expression~\eqref{eq:phi_rat} we see that any monic multiple of a recurrence polynomial $g_i$ is again a recurrence polynomial.
Therefore we may replace $g_1,\ldots,g_n$ by their product $g=g_1\cdots g_n$, \ul{and assume without loss of generality that $\gamma$ has the same recurrence polynomial $g$ in every dimension.}

\medskip

Below we give some examples of multi-dimensional linear recurrence sequences.

\begin{example}\label{example:multi_dim_LRS}
    Let $R$ be a ring and $M$ be an $R$-module.
    \begin{enumerate}[label=(\roman*), topsep=0pt, itemsep=0em, leftmargin=0.7cm]
        \item
        Let $\alpha, \beta \colon \N \rightarrow M$ be one-dimensional linear recurrence sequences.
        Then
        \[
        \gamma(a_1, a_2) \coloneqq \alpha(a_1) + \beta(a_2)
        \]
        is a 2-dimensional linear recurrence sequence.
        Indeed, $\gamma$ satisfies the same recurrence relation as $\alpha$ in the direction of $a_1$, and the same recurrence relation as $\beta$ in the direction of $a_2$.

        We now compute the associated power series $\phi_\gamma$ in terms of $\alpha$ and $\beta$.
        Let $g_{\alpha}, g_{\beta}$ be the recurrence polynomials of $\alpha, \beta$ and write their associated power series as $\phi_{\alpha}(T) = \frac{h_{\alpha}(T)}{g_{\alpha}^{\circ}(T)}, \phi_{\beta}(T) = \frac{h_{\beta}(T)}{g_{\beta}^{\circ}(T)}$.
        Then
        \begin{align*}\label{eq:example_sum_LRS}
        \phi_{\gamma}(T_1, T_2) & = \sum_{a_1, a_2 \in \N} (\alpha(a_1) + \beta(a_2))T_1^{a_1}T_2^{a_2}
         = \sum_{a_1 \in \N} \alpha(a_1) T_1^{a_1} \sum_{a_2 \in \N} T_2^{a_2} + \sum_{a_1 \in \N} T_1^{a_1} \sum_{a_2 \in \N} \beta(a_2) T_2^{a_2}  \nonumber\\
        & = \phi_{\alpha}(T_1) \cdot \frac{1}{1 - T_2} + \frac{1}{1 - T_1} \cdot \phi_{\beta}(T_2) 
         = \frac{h_{\alpha}(T_1)}{g_{\alpha}^{\circ}(T_1)(1 - T_2)} + \frac{h_{\beta}(T_2)}{(1 - T_1)g_{\beta}^{\circ}(T_2)} \nonumber\\
        & = \frac{H(T_1, T_2)}{(1 - T_1)g_{\alpha}^{\circ}(T_1) \cdot (1 - T_2)g_{\beta}^{\circ}(T_2)},
        \end{align*}
        where $H(T_1, T_2) = (1-T_1)g_{\beta}^{\circ}(T_2)h_{\alpha}(T_1) + (1-T_2)g_{\alpha}^{\circ}(T_1)h_{\beta}(T_2)$. 
        We can take the recurrence polynomial of $\gamma$ (in both dimensions) to be any common multiple of $(T - 1)g_{\alpha}(T)$ and $(T - 1)g_{\beta}(T)$, such as $(T-1)g_{\alpha}(T)g_{\beta}(T)$.

        Note that if we replace $\beta$ by $-\beta$, then deciding whether $\mZ(\gamma)$ is nonempty is equivalent to determining whether the sets $\{\alpha(k) \mid k \in \N\}$ and $\{\beta(m) \mid m \in \N\}$ have a common element.

        \item The above example generalizes to the sum of any $n \geq 2$ sequences.
        In their second joint paper, Derksen and Masser~\cite{derksen2015linear} showed that the set $\{(a_1, \ldots, a_n) \in \N^n\mid \sum_{i=1}^n \gamma_i(a_i)=0\}$, where $\gamma_1(a),\ldots,\gamma_n(a)$ are recurrence sequences over a field of characteristic $p$, is $p$-normal in $\N^n$. 
        We observe that $\gamma(a_1, \ldots, a_n) \coloneqq \sum_{i=1}^n \gamma_i(a_i)$ is an $n$-dimensional recurrence sequence in our sense.
        Therefore, Theorem~\ref{thm:LRS} generalizes~\cite{derksen2015linear} in two ways: it works on rings and modules instead of fields, and accepts more general recurrence sequences.

        \item 
        Generalized linear-exponential maps provide natural examples of $n$-dimensional linear recurrence sequences.
        Let $\bs = (s_1, \ldots, s_n)$ be a tuple of \emph{not necessarily invertible} elements in $R$ and let $m \in M$.
        Define $\gamma(a_1, \ldots, a_n) \coloneqq s_1^{a_1} \cdots s_n^{a_n} m$ for $(a_1, \ldots, a_n) \in \N^n$, then its associated power series is 
        \begin{multline*}
        \phi_{\gamma}(T_1, \ldots, T_n)
         = \sum_{(a_1, \ldots, a_n) \in \N^n}  s_1^{a_1} \cdots s_n^{a_n} m \cdot T_1^{a_1} \cdots T_n^{a_n}
         = \sum_{a_1 = 0}^{\infty} \cdots \sum_{a_d = 0}^{\infty}  (s_1T_1)^{a_1} \cdots (s_nT_n)^{a_n} m \\
         = \frac{m}{(1 - s_1T_1) \cdots (1 - s_nT_n)}.
        \end{multline*}
        We can see that $\phi_{\gamma}$ satisfies the form~\eqref{eq:phi_rat}, therefore $\gamma$ is a linear recurrence sequence of dimension $n$.
        Since $\phi_{\gamma_1 + \gamma_2} = \phi_{\gamma_1} + \phi_{\gamma_2}$, we can see that if $\gamma_1$ and $\gamma_2$ are $n$-dimensional linear recurrence sequences, then so is $\gamma_1 + \gamma_2$.
        Therefore any linear-exponential sum $\gamma(\ba) = \sum_{i=1}^k \bs_i^{\ba} m_i$ where $\bs_1, \ldots, \bs_k \in R^n, m_1, \ldots, m_k \in M$ is an $n$-dimensional linear recurrence sequence over the $R$-module $M$.

        \item
        We compute an explicit example which will appear in the proof of Theorem~\ref{thm:LRS}.
        For $\ell \geq 1$ and $s_1, \ldots, s_{\ell} \in R$, define
        \[
        \gamma(a_1, \ldots, a_n) \coloneqq \sum_{(i_1, \ldots, i_n) \in \{1, \ldots, \ell\}^n} s_{i_1}^{a_1} \cdots s_{i_n}^{a_n} \cdot m_{i_1, \ldots, i_n},
        \]
        where $m_{i_1, \ldots, i_n} \in M$.
        That is, $\gamma$ is the linear combination of all possible $n$-fold exponential products with bases among $\{s_1, \ldots, s_{\ell}\}$.
        Then its associated power series is
        \begin{align*}
        \phi_{\gamma}(T_1, \ldots, T_n) & = \sum_{(a_1, \ldots, a_n) \in \N^n} \sum_{(i_1, \ldots, i_n) \in \{1, \ldots, \ell\}^n} s_{i_1}^{a_1} \cdots s_{i_n}^{a_n} \cdot m_{i_1, \ldots, i_n} T_1^{a_1} \cdots T_n^{a_n} \\
        & = \sum_{(i_1, \ldots, i_n) \in \{1, \ldots, \ell\}^n} \frac{m_{i_1, \ldots, i_n}}{(1-s_{i_1}T_1) \cdots (1-s_{i_n}T_n)} \\
        & = \frac{\sum_{(i_1, \ldots, i_n) \in \{1, \ldots, \ell\}^n} m_{i_1, \ldots, i_n} \prod_{j = 1}^n \prod_{1 \leq i \leq \ell, i \neq i_j} (1-s_{i}T_j)}{\prod_{j = 1}^n (1-s_1T_j) (1-s_2T_j) \cdots (1-s_{\ell}T_j)}.
        \end{align*}
        This satisfies the form~\eqref{eq:phi_rat} since the denominator has degree at most $\ell-1$ in each $T_j$.
        Therefore $\gamma$ has the recurrence polynomial $g(T) = ((1-s_1T_j) (1-s_2T_j) \cdots (1-s_{\ell}T_j))^{\circ} = (T - s_1)(T - s_2) \cdots (T - s_{\ell})$ in each of the $n$ dimensions.
        \hfill $\blacksquare$
    \end{enumerate}
\end{example}
 
\subsection{p-normal sets in $\N^n$}\label{subsec:pnormal_N}
We define $p$-normal sets in $\mathbb N^n$ using a simplified reformulation of~\cite{derksen2015linear}.

\begin{definition}[$p$-normal sets in $\N^n$]\label{def:pnormal_N}
    For $t \in \N$, let
    $
    \mP(t) = \big\{ \{0\}, \{1\}, \cdots \{t-1\}, \{a \in \N \mid a \geq t\} \big\}.
    $
    Thus $\mP(t)$ is the partition of $\N$ into singleton sets $\{0\}, \ldots, \{t-1\}$ together with the tail $\{a \in \N \mid a \geq t\}$.
    Given $t_1, \ldots, t_n \in \N$, the corresponding \emph{grid partition} of $\N^n$ is
    \[
    \N^n = \bigcup_{A_1 \in \mP(t_1)} \cdots \bigcup_{A_n \in \mP(t_n)} (A_1 \times \cdots \times A_n).
    \]
    For example, taking $t_1 = t_2 = 1$ gives the partition of $\N^2$ into four cells
    $\{(0,0)\}$, $\{0\} \times \{a \geq 1\}$, $\{a \geq 1\} \times \{0\}$, $\{a \geq 1\} \times \{a \geq 1\}$.
    
    A set $S \subseteq \N^n$ is called \emph{$p$-normal in $\N^n$}, if there exists a grid partition as above, such that
    \[
    S = \bigcup_{A_1 \in \mP(t_1)} \cdots \bigcup_{A_n \in \mP(t_n)} \big(T_{A_1, \ldots, A_n} \cap (A_1 \times \cdots \times A_n)\big),
    \]
    where each $T_{A_1, \ldots, A_n}$ is a $p$-normal set in $\Z^n$.
    In other words, $S$ is $p$-normal in $\N^n$ if, on each cell of a suitable grid partition, it coincides with the restriction of a $p$-normal set in $\Z^n$.
\end{definition}

Since the intersection of two $p$-normal sets in $\Z^n$ is $p$-normal in $\Z^n$ (Lemma~\ref{lem:inter_normal}), by considering each grid cell we immediately obtain that the intersection of two $p$-normal sets in $\N^n$ is $p$-normal in $\N^n$.
Furthermore, it is easy to see that $p$-normality is preserved under scaling and translation, i.e., if $S$ is $p$-normal in $\N^n$, then $q \cdot S$ and $S + \bb$ are $p$-normal in $\N^n$ for $q \in \N_{>0}$ and $\bb \in \N^n$.

Remark that $p$-normal sets in $\N^n$ are definable in the existential first-order theory of the structure $\gen{\Z; 0, 1, <, +, p^{\N}}$, where $p^{\N}$ is the set of positive integer powers of $p$ (see~\cite{karimov2025decidability} for a detailed definition).\footnote{In our definition of $p$-normal sets of $\Z^n$ (Definition~\ref{def:pnormal}), it is necessary to express the set $\{p^{\ell k} \mid k \in \N\}$ for any $\ell \geq 1$. This can be done using the predicate $z \in p^{\N}$ together with the modular constraint $z \equiv 1 \pmod{p^{\ell} - 1}$, which are expressible in the existential fragment of $\gen{\Z; 0, 1, +, p^{\N}}$.}
Therefore, deciding emptiness of the intersection $\bigcap_{i = 1}^k S_{p_i}$ where each $S_{p_i}$ is $p_i$-normal in $\N^n$, reduces to the existential first-order theory of the structure $\gen{\Z; 0, 1, <, +, p_1^{\N}, \ldots, p_k^{\N}}$.
This is a difficult problem in general.
The case of $k = 1$ is decidable following the classical work of B\"{u}chi and Sem\"{e}nov~\cite{buchi1960weak, semenov1980certain}.
The case of $k=2$ was recently proved decidable by Karimov, Luca, Nieuwveld, Ouaknine and Worrell~\cite{karimov2025decidability} using Baker's theorem; see~\cite{bajpai2023effective, hieronymi2022strong, HieronymiReitmeirWang2026} for related results.
The case of $k \geq 3$ already subsumes longstanding open problems in number theory, such as finding all the finitely many non-negative integer solutions to the Diophantine equation $3^a + 5^b - 7^c = 1$.

Note that in the special case of dimension $n = 1$, the emptiness of $\bigcap_{i = 1}^k S_{p_i}$ is known to be decidable by a result of Dong and Shafrir~\cite{dong2025skolemproblemringspositive} (see Theorem~\ref{thm:LRS_ring}).

\subsection{Proof of multi-dimensional Skolem-Mahler-Lech}

In this section we give the proof of:
\thmLRS*

We need some preparations.
Recall that we can without loss of generality suppose $\gamma$ to have the same recurrence polynomial $g(T)$ in every dimension.
By definition, $g$ is monic (i.e., its leading coefficient is $1$).
We say that a monic polynomial \emph{splits} over a ring $R$, if it can be factorized into a product $(T-r_1) \cdots (T - r_d)$ with $r_1, \ldots, r_d \in R$, not necessarily distinct.
First, we need to extend the ring $R$ (and correspondingly, the module $M$), so that $g$ splits over $R$.

\begin{lemma}[folklore]\label{lem:split_extension}
    Let $R$ be a ring, $M$ be an $R$-module, and $g \in R[T]$ be a monic polynomial. We can compute a ring extension $\widetilde R \supseteq R$, finitely generated as an $R$-module, such that $g$ splits over $\widetilde R$, and such that the map $M\rightarrow M\otimes _R \widetilde R$ is injective. 
\end{lemma}
\begin{proof}
    We prove by induction on the degree $d \coloneqq \deg(g)$. 
    Define $\widetilde R$ to be the quotient $R[x]/\langle g(x)\rangle$ of the polynomial ring $R[x]$. 
    As $R$-modules, $\widetilde R = R + xR + \cdots + x^{d-1}R \cong R^d$, and therefore $\widetilde M =M\otimes_R \widetilde R \cong M^d$, and the map $M\rightarrow M\otimes _R \widetilde R$ is injective. 
    In the ring $\widetilde R[T]$, we have $T - x \mid g(T)$.
    Therefore we can let $\widetilde g(T) \coloneqq \frac{g(T)}{T-x}$ and use induction on $\widetilde R, \widetilde M, \widetilde g$. 
\end{proof}

Therefore we can extend $R,M$, and without loss of generality suppose the recurrence polynomial of $\gamma$ splits over $R$.
Next, we need the following analogue of Proposition~\ref{prop:decompose_localize} for linear recurrence sequences, in order to reduce to the case of modules over local rings:
\begin{lemma}\label{lem:decompose_LRS}
    Let $\gamma \colon \N^n\rightarrow M$ be a linear recurrence sequence over an $R$-module $M$.
    Then one can effectively write $\mZ(\gamma) = \bigcap_{i = 1}^d \mZ(\gamma_i)$, where each $\gamma_i:\N^n\rightarrow M_i$ is a linear recurrence sequence over a finitely generated $R_i$-module $M_i$, with $R_i$ local with maximal ideal $\frm_i$, and $\frm_i^t = 0$ for some $t \in \N$.
    Moreover, the recurrence polynomial of $\gamma_i$ splits over $R_i$.
    Furthermore, if $\ell M = 0$ for some $\ell \in \N_{>0}$, then for each $i = 1, \ldots, d$, the residue field $R_i/\frm_i$ has characteristic $p_i$ for some prime number $p_i$ dividing $\ell$.
\end{lemma}
\begin{proof}
    By restricting $R$ to the subring generated by the coefficients $c_{ij}$ in the recurrences~\eqref{eq:recurrence_def}, and restricting $M$ to the submodule generated by the initial values of $\gamma$, we can suppose $R$ to be finitely generated and $M$ to be finitely generated over $R$.
    By Lemma~\ref{lem:split_extension} we can suppose that the recurrence polynomial of $\gamma$ splits over $R$.
    The rest of the proof is essentially the same as Proposition~\ref{prop:decompose_localize}, replacing linear-exponential maps by $n$-dimensional linear recurrence sequences.
    Using primary decomposition and localization, we find maps $\varphi_i \colon R \rightarrow R_i, \phi_i \colon M \rightarrow M_i, i = 1, \ldots, d$, such that $\mZ(\gamma) = \bigcap_{i = 1}^d \mZ(\phi_i \circ \gamma)$, and such that $R_i, M_i$ satisfy the requirements in the lemma.
    Then $\gamma_i \coloneqq \phi_i \circ \gamma$ is a linear recurrence sequence in the $R_i$-module $M_i$, whose recurrence polynomial is obtained by applying $\varphi_i$ coefficientwise to the recurrence polynomial of $\gamma$.
    Hence, this polynomial splits over $R_i$.
\end{proof}

By Lemma~\ref{lem:decompose_LRS}, it suffices to focus on the case where $R$ is local, the maximal ideal $\frm$ contains some prime number $p$, and $\frm^t = 0$ for some $t \in \N$.
We can also suppose that the recurrence polynomial of $\gamma$ splits as $g(T) = \prod_{i = 1}^d(T-r_i)$, where the roots $r_1, \ldots, r_d \in R$ are not necessarily distinct.
The following proposition shows that, by taking a suitable multiple of $g$ and after a suitable variable change, we can suppose the roots to be distinct even modulo $\frm$.

\begin{proposition}\label{prop:seperable_pol}
    Let $r_1,...,r_d\in R$, not necessarily distinct, and $\frm \subseteq R$ be a prime ideal containing some prime number $p$, and $\frm^t=0$ for some $t \in \N$. Then there exist effectively computable $s_1,...,s_{\ell}\in R$, pairwise distinct modulo $\frm$, and $u\in\N$, such that:
    \begin{equation}\label{eq:sep_pol}
    \prod_{i = 1}^d(T-r_i) \;\Big|\; \prod_{i = 1}^{\ell} \left(T^{p^u}-s_i \right).
    \end{equation}
\end{proposition}
\begin{proof}
    First consider the case where all $r_i$ are equivalent modulo $\frm$. 
    By Lemma~\ref{lem:lifting}, $r_i\equiv r_1 \pmod{\frm}$ implies $r_i^{p^t}\equiv r_1^{p^t} \pmod{\frm^{t+1}=0}$, so $r_i^{p^t} = r_1^{p^t}$.  
    Now, $T-r_i$ divides $T^{p^t}-r_i^{p^t}=T^{p^t}-r_1^{p^t}$, and therefore:
    \[
    \prod_{i = 1}^d(T-r_i) \;\Big|\; \left(T^{p^t}-r_1^{p^t} \right)^d.
    \]
    To obtain the polynomial on the right hand side of~\eqref{eq:sep_pol}, we need to replace $\left(T^{p^t}-r_1^{p^t} \right)^d$ with $T^{p^u}-r_1^{p^u}$ for some $u$.
    This boils down to proving the following claim:
    \medskip
    \begin{adjustwidth}{3mm}{3mm}
    \textbf{Claim.} Let $n = t+d$, then for any $s \in R$, we have
    \[
    (T-s)^d\mid T^{p^n}-s^{p^n}.
    \]
    \medskip
    \noindent$\blacktriangleright$
    Define the ideal $I=\gen{T-s, \frm}$ of $R[T]$. Since $T\equiv s \pmod{I}$ and $p\in I$, using Lemma~\ref{lem:lifting} we get $T^{p^n}\equiv s^{p^n} \pmod{I^{n+1}}$ for every $n$.
    Note that every element in $I^{n+1}$ is an $R[T]$-linear combination of elements of the form $f_{1} \cdots f_{k} \cdot (T - s)^{n+1-k}$ with $0 \leq k \leq n+1$, $f_1, \ldots, f_k \in \frm$.
    Since $\frm^t=0$, this element is either $0$ or it is divisible by $(T-s)^{n+2-t}$.
    Therefore $I^{n+1} \subseteq \gen{(T-s)^{n+2-t}}$ for all $n \geq t$.
    Since $T^{p^n}\equiv s^{p^n} \pmod{I^{n+1}}$, we have $(T-s)^{n+2-t} \mid T^{p^n}-s^{p^n}$.
    Taking $n = d+t$ gives
    $
    (T-s)^{d+2}\mid T^{p^n}-s^{p^n}
    $, proving the claim.
    \hfill $\blacktriangleleft$
    \end{adjustwidth}
    \medskip

\noindent In the claim, replacing $T$ with $T^{p^t}$ and $s$ with $r_1^{p^t}$, we get altogether: 
\[
\prod_{i = 1}^d (T-r_i) \;\Big|\; \left(T^{p^t}-r_1^{p^t} \right)^d \;\bigg|\; T^{p^{t+n}}-r_1^{p^{t+n}}.
\]
This proves the proposition for the case where all $r_i$ are equivalent modulo $\frm$. 

For the general case, choose $\mI \subseteq\{1,...,d\}$ so that $\{r_i\mid i\in \mI\}$ has a single representative from each class modulo $\frm$. 
Using the previous case, we obtain a $u \in \N$ such that
\[
\prod_{i = 1}^d (T-r_i) \;\Big|\; \prod_{i\in \mI} \left(T^{p^u}-r_i^{p^u} \right).
\]
We now show that $r_i^{p^u}, i\in \mI$, are distinct modulo $\frm$.
We claim that for $a, b \in R$, $a \not\equiv b \pmod{\frm} \implies a^p\not\equiv b^p \pmod{\frm}$.
Indeed, if $a-b\notin \frm$, then $(a-b)^p\notin\frm$ since $\frm$ is prime. Since $\frm$ contains the prime number $p$, we get $(a-b)^p\equiv a^p-b^p \pmod{\frm}$, so $a^p\not\equiv b^p \pmod{\frm}$, proving the claim. 
Consider two different elements $i, j \in \mI$, using the claim repeatedly we obtain $r_i \not\equiv r_j \pmod{\frm} \implies r_i^p \not\equiv r_j^p \pmod{\frm} \implies \cdots \implies r_i^{p^u}\not\equiv r_j^{p^u} \pmod{\frm}$, for all $u \geq 0$.

Finally, we obtain the proposition by letting $s_i \coloneqq r_i^{p^u}$.
\end{proof}

As an example for the statement of Proposition~\ref{prop:seperable_pol}, it is easy to verify that in $\Z_{/4}[T]$, the polynomial $(T-1)^2(T-2)$ divides $(T^4-1)T^4$.
\medskip

We are now ready to prove Theorem~\ref{thm:LRS}:
\begin{proof}[Proof of Theorem~\ref{thm:LRS}]
We can write $\mZ(\gamma)$ as a finite intersection $\bigcap_{i = 1}^d \mZ(\gamma_i)$ as in Lemma~\ref{lem:decompose_LRS}, and it suffices to show that each $\mZ(\gamma_i)$ is $p_i$-normal, since $p$-normal sets are closed under finite intersections.
Thus we are reduced to the case where $R$ is a local ring with maximal ideal $\frm$, with $\frm^t = 0$ for some $t \in \N$, the residue field $R/\frm$ has characteristic $p$, and such that the recurrence polynomial $g$ of the sequence $\gamma$ splits over $R$.

By Proposition~\ref{prop:seperable_pol}, there are $\ell, u \in \N$ and $s_1,\ldots,s_{\ell} \in R$, distinct modulo $\frm$, such that $g(T)\mid \prod_{i=1}^{\ell} (T^{p^u}-s_i)$.

    \medskip
    \begin{adjustwidth}{3mm}{3mm}
    \textbf{Claim.} For each $\bj \in \{0, 1, \ldots, p^u-1\}^n$, the sequence $\gamma_{\bj}(\ba) \coloneqq \gamma(p^u \ba + \bj)$ admits the recurrence polynomial $\prod_{i=1}^{\ell} (T-s_i)$ in every dimension.
    \medskip
    
    \noindent$\blacktriangleright$
    Let $\tilde g(T) \coloneqq \prod_{i=1}^{\ell} (T-s_i)$, then $g(T) \mid \tilde g(T^{p^u})$, so $\tilde g(T^{p^u})$ is also a recurrence polynomial for $\gamma$.
    Write $\tilde g(T^{p^u}) = T^{p^u \ell} - c_1 T^{p^u(\ell - 1)} - \cdots - c_{\ell}$, then $\gamma$ satisfies the recurrence 
    \[
    \gamma(\ba) = c_1 \cdot \gamma(a_1, \ldots, a_i-p^u, \ldots, a_n) + \cdots + c_{\ell} \cdot \gamma(a_1, \ldots, a_i-p^u\ell, \ldots, a_n)
    \]
    for every $i$ and $a_i \geq p^u \ell$.
    Therefore the sequence $\gamma_{\bj}(\ba) \coloneqq \gamma(p^u \ba + \bj)$ has the recurrence polynomial $T^{\ell} - c_1 T^{\ell - 1} - \cdots - c_{\ell} = \tilde g(T)$ in every dimension.
    \hfill $\blacktriangleleft$
    \end{adjustwidth}
    \medskip
If we prove that each $\mZ(\gamma_{\bj})$ is $p$-normal, then we obtain that $\mZ(\gamma) = \bigcup_{\bj \in \{0, 1, \ldots, p^u-1\}^n} \left(p^u \mZ(\gamma_{\bj}) + \bj \right)$ is also $p$-normal.
Therefore without loss of generality we may replace $\gamma$ by $\gamma_{\bj}$ and assume the recurrence polynomial of $\gamma$ is $\prod_{i=1}^{\ell} (T-s_i)$, with $s_1, \ldots, s_{\ell} \in R$, distinct modulo $\frm$.

    \medskip
    \begin{adjustwidth}{3mm}{3mm}
    \textbf{Claim.} We can write $\gamma$ as a linear-exponential sum 
    \begin{equation}\label{eq:alpha_lin_exp}
    \gamma(a_1, \ldots, a_n) = \sum_{(i_1, \ldots, i_n) \in \{1, \ldots, \ell\}^n} s_{i_1}^{a_1} \cdots s_{i_n}^{a_n} \cdot m_{i_1, \ldots, i_n},
    \end{equation}
    for some effectively computable $m_{i_1, \ldots, i_n} \in M,\; (i_1, \ldots, i_n) \in \{1, \ldots, \ell\}^n$.
    \medskip
    
    \noindent$\blacktriangleright$
    We first restrict to the 1-dimensional case for simplicity.
    Define $V \in R^{\ell \times \ell}$ to be the Vandermonde matrix associated with $s_1,\ldots,s_{\ell} \in R$.
    That is, $V$ is the $\ell \times \ell$ matrix whose $i$-th row is $(1,s_i,s_i^2,\ldots, s_i^{\ell-1})$.
    Then $\det(V)=\prod_{1\leq i<j \leq \ell}(s_i-s_j)$~\cite[Section~0.9.11]{horn2012matrix}. 
    Recall that $R$ is local with maximal ideal $\frm$.
    For $i < j$, since $s_i \not\equiv s_j \pmod{\frm}$, we have $s_i - s_j \in R \setminus \frm = R^{\times}$.
    Therefore $\det(V) \in R^{\times}$ and consequently $V$ is invertible in $R^{\ell \times \ell}$.

    Since $V$ is invertible, we can compute a row vector $(m_1, \ldots, m_{\ell}) \in M^{\ell}$ such that\footnote{Here we are considering $M$ as a right $R$-module, which is the same as our earlier convention of $M$ as a left module since $R$ is commutative.}
    \[
    (m_1, \ldots, m_{\ell} ) \cdot V = (\gamma(0),...,\gamma(\ell-1) ).
    \]
    This means that
    \[
    \textstyle
    \sum_{i=1}^{\ell}s_i^a m_i = \gamma(a),\; \text{ for } a = 0, 1, \ldots, \ell-1.
    \]
    Therefore the two sequences $\gamma(a)$ and $\tilde \gamma(a) \coloneqq \sum_{i=1}^{\ell}s_i^a m_i$ have the same $\ell$ initial values.
    Furthermore, they satisfy the same recurrence polynomial $\prod_{i=1}^{\ell} (T-s_i)$ of degree $\ell$ (see Example~\ref{example:multi_dim_LRS}(iv)).
    This implies that $\gamma(a)$ and $\tilde \gamma(a)$ are the same sequence, so $\gamma(a)=\sum_{i=1}^{\ell}s_i^a m_i$ for all $a \in \N$.
    This concludes the 1-dimensional case.
    
    The $n$-dimensional case is similar. 
    Let $V^{\otimes n} \in R^{\ell^n \times \ell^n}$ be the $n$-fold tensor power of the Vandermonde matrix $V$ associated with $s_1,\ldots,s_{\ell}$.
    That is, $V^{\otimes n}$ is the $\ell^n \times \ell^n$ matrix whose rows and columns are indexed by the set $\{1, \ldots, \ell\}^n$, and whose $\big((i_1, \ldots, i_n), (j_1, \ldots, j_n)\Big)$-entry is $s_{i_1}^{j_1-1} s_{i_2}^{j_2-1} \cdots s_{i_n}^{j_n-1}$.
    Since $V$ has an inverse $V^{-1}$, the tensor product $V^{\otimes n}$ has an inverse $(V^{-1})^{\otimes n}$ in $R^{\ell^n \times \ell^n}$ (see~\cite[p.408]{LancasterTismenetsky1985}).

    Since $V^{\otimes n}$ is invertible, we can compute a row vector $\big(m_{(i_1, \ldots, i_n)}\big)_{(i_1, \ldots, i_n) \in \{1, \ldots, \ell\}^n} \in M^{\ell^n}$ such that
    \[
    \Big(m_{(i_1, \ldots, i_n)}\Big)_{(i_1, \ldots, i_n) \in \{1, \ldots, \ell\}^n} \cdot V^{\otimes n} =  \Big(\gamma(j_1-1, \ldots, j_n-1)\Big)_{(j_1, \ldots, j_n) \in \{1, \ldots, \ell\}^n}.
    \]
    This means that
    \[
    \sum_{(i_1, \ldots, i_n) \in \{1, \ldots, \ell\}^n} s_{i_1}^{a_1} \cdots s_{i_n}^{a_n} \cdot m_{i_1, \ldots, i_n} = \gamma(a_1, \ldots, a_n),\; \text{ for } (a_1, \ldots, a_n) \in \{0, 1, \ldots, \ell-1\}^n.
    \]
    Therefore the sequences on both sides have the same initial values and the same recurrence polynomial $\prod_{i=1}^{\ell} (T-s_i)$ of degree $\ell$ (see Example~\ref{example:multi_dim_LRS}(iv)).
    Just as in the 1-dimensional case we obtain that they are the same sequence.
    \hfill $\blacktriangleleft$
    \end{adjustwidth}
    \medskip

\noindent It now suffices to show that the zero set $\mZ(\gamma) \subseteq \N^n$ of the linear-exponential sum~\eqref{eq:alpha_lin_exp} is $p$-normal in $\N^n$.
If none of the elements $s_1, \ldots, s_{\ell}$ is in the maximal ideal $\frm$, then they are all invertible and hence the linear-exponential sum~\eqref{eq:alpha_lin_exp} is well-defined over $\Z^n$.
Therefore by Theorem~\ref{thm:main} (or Theorem~\ref{thm:local}), the zero set $\mZ(\gamma)$ is the intersection of $\N^n$ with a $p$-normal subset of $\Z^n$, which is $p$-normal in $\N^n$.

If some of the elements $s_1, \ldots, s_{\ell}$ are in the maximal ideal $\frm$, since they are distinct modulo $\frm$, we may without loss of generality suppose $s_1, \ldots, s_{\ell-1} \notin \frm, s_{\ell} \in \frm$.
Recall that $\frm^t = 0$, so $s_{\ell}^t=0$.
By partitioning each coordinate $a_i, i \in \{1, \ldots, n\}$ according to whether
$a_i = 0, a_i = 1, \ldots, a_i = t-1$, or $a_i \geq t$, we obtain a grid partition of $\N^n$ (see Definition~\ref{def:pnormal_N}).
On each grid cell $L$, the restriction of $\gamma$ can be represented by a linear-exponential map $\gamma_L \colon \Z^n \rightarrow M$, obtained by fixing the variables $a_i$ specified by $a_i = \star$, and removing the terms containing the factor $s_\ell^{a_i}$ specified by $a_i\ge t$. 
Hence,
$
\mZ(\gamma)=\bigcup_L \bigl(\mZ(\gamma_L)\cap L\bigr),
$
where $L$ ranges over the cells of the grid partition.
Note that each $\gamma_{L}$ is now defined over $\Z^n$, so by Theorem~\ref{thm:main}, $\mZ(\gamma_L)$ is $p$-normal in $\Z^n$.
By Definition~\ref{def:pnormal_N} we obtain that $\mZ(\gamma)$ is $p$-normal in $\N^n$.
\end{proof}

The last step of the above proof can be illustrated by the following example.
\begin{example}
    Let $R$ be a local ring with maximal ideal $\frm$, such that $\frm^2 = 0$.
    Consider a linear-exponential sum
    \[
    \gamma(a_1, a_2) = s_1^{a_1} s_1^{a_2} m_{11} + s_1^{a_1} s_2^{a_2} m_{12} + s_2^{a_1} s_1^{a_2} m_{21} + s_2^{a_1} s_2^{a_2} m_{22},
    \]
    where $s_1 \notin \frm, s_2 \in \frm$.
    Then $s_1$ is invertible, and $s_2^2 = 0$.
    The zero set $\mZ(\gamma)$ is decomposed into $9$ components, according to whether $a_1 = 0, a_1 = 1$, or $a_1 \geq 2$, and whether $a_2 = 0, a_2 = 1$, or $a_2 \geq 2$.
    \begin{enumerate}[label = (\roman*)]
        \item If $a_1 \geq 2$ and $a_2 \geq 2$, then $s_2^{a_1} = s_2^{a_2} = 0$, so
        \begin{align*}
            \gamma(a_1, a_2) & = s_1^{a_1} s_1^{a_2} m_{11} + s_1^{a_1} s_2^{a_2} m_{12} + s_2^{a_1} s_1^{a_2} m_{21} + s_2^{a_1} s_2^{a_2} m_{22} \\
            & = s_1^{a_1} s_1^{a_2} m_{11}.
        \end{align*}
        In this case, we set $\tilde \gamma(a_1, a_2) \coloneqq s_1^{a_1} s_1^{a_2} m_{11}$, which is well-defined over $\Z^2$.
        This gives the grid component $\mZ(\tilde\gamma) \cap \big(\{a_1 \geq 2\} \times \{a_2 \geq 2\}\big)$.
        \item If $a_1 \geq 2$ and $a_2 = 1$, then $s_2^{a_1} = 0, s_1^{a_2} = s_1, s_2^{a_2} = s_2$, so
        \begin{align*}
            \gamma(a_1, a_2) & = s_1^{a_1} s_1^{a_2} m_{11} + s_1^{a_1} s_2^{a_2} m_{12} + s_2^{a_1} s_1^{a_2} m_{21} + s_2^{a_1} s_2^{a_2} m_{22} \\
            & = s_1^{a_1} \cdot s_1 m_{11} + s_1^{a_1} \cdot s_2 m_{12}.
        \end{align*}
        We can multiply it with the dummy factor $1^{a_2}$ and let $\tilde\gamma(a_1, a_2) \coloneqq s_1^{a_1} 1^{a_2} \cdot s_1 m_{11} + s_1^{a_1} 1^{a_2} \cdot s_2 m_{12}$.
        This gives the grid component $\mZ(\tilde\gamma) \cap \big(\{a_1 \geq 2\} \times \{1\}\big)$.
        The case where $a_1 \geq 2, a_2 = 0$, and the cases where $a_1 \in \{0, 1\}, a_2 \geq 2$ are analogous.
        \item If $a_1 = 1$ and $a_2 = 1$, then $\gamma(a_1, a_2)$ becomes the constant function $s_1 s_1 m_{11} + s_1 s_2 m_{12} + s_2 s_1 m_{21} + s_2 s_2 m_{22}$. We obtain the grid component $\{(1, 1)\}$ or $\emptyset$. The other cases are analogous.
    \end{enumerate}
    Finally, $\mZ(\gamma)$ is the union of the $9$ grid components obtained above.
    \hfill $\blacksquare$
\end{example}

\section*{Acknowledgements} The authors would like to thank Joris Nieuwveld and James Worrell for discussion about the Skolem Problem. Ruiwen Dong is supported by a Fellowship by Examination at Magdalen College, Oxford.

\bibliography{short}

\end{document}